\documentclass[aihp]{imsart}

\RequirePackage[dvipsnames]{xcolor}
\RequirePackage{amsthm,amsmath,amsfonts,amssymb}
\RequirePackage[numbers]{natbib}
\RequirePackage[colorlinks,citecolor=blue,urlcolor=blue]{hyperref}
\RequirePackage{graphicx}

\RequirePackage{tikz}
\usetikzlibrary{arrows}
\RequirePackage{enumitem}
\RequirePackage{bbm}
\RequirePackage{multirow,stackengine}
\startlocaldefs
\numberwithin{equation}{section}
\theoremstyle{plain}
\newtheorem{theorem}{Theorem}[section]
\newtheorem{lemma}[theorem]{Lemma}
\newtheorem{proposition}[theorem]{Proposition}

\theoremstyle{remark}
\newtheorem{definition}[theorem]{Definition}
\newtheorem{rmk}[theorem]{Remark}

\newcommand{\ee}{\end{equation}}
\newcommand{\de}{\begin{align*}}
\newcommand{\fe}{\end{align*}}
\newcommand{\R}{\mathbb{R}}

\newcommand{\E}{\mathbb{E}}

\newcommand{\N}{\mathbb{N}}
\newcommand{\Z}{\mathbb{Z}}
\newcommand{\mmin}{\mathfrak{min}}

\newcommand{\red}[1]{\textcolor{red}{#1}}
\newcommand{\blue}[1]{\textcolor{blue}{#1}}

\newcommand{\olive}[1]{\textcolor{olive}{#1}}
\newcommand{\purple}[1]{\textcolor{purple}{#1}}
\newcommand{\q}{t}

\newcommand{\LL}{\mathsf{L}}
\newcommand{\A}{\mathsf{A}}
\newcommand{\e}{\epsilon}
\newcommand{\mfs}{\mathfrak{s}}
\newcommand{\mfi}{\mathfrak{i}}
\newcommand{\mfj}{\mathfrak{j}}
\newcommand{\mfk}{\mathfrak{k}}
\newcommand{\mfl}{\mathfrak{l}}
\newcommand{\bx}{\mathbf{x}}
\newcommand{\bl}{\mathbf{l}}
\newcommand{\bk}{\mathbf{k}}
\newcommand{\bi}{\mathbf{i}}
\newcommand{\bj}{\mathbf{j}}
\newcommand{\h}{h}

\endlocaldefs

\begin{document}

\begin{frontmatter}

\title{Hydrodynamics of the $t$-PNG model via a colored $t$-PNG model}
\runtitle{Hydrodynamics of the $t$-PNG model}

\begin{aug}
\author[A]{\fnms{Hindy}~\snm{Drillick}\ead[label=e1]{hindy.drillick@columbia.edu}\orcid{0000-0003-1515-9922}    }
\and
\author[B]{\fnms{Yier}~\snm{Lin}\ead[label=e2]{ylin10@uchicago.edu}\orcid{0000-0001-8197-0516}}
\address[A]{Department of Mathematics, Columbia University\printead[presep={,\ }]{e1}}

\address[B]{Department of Statistics, University of Chicago\printead[presep={,\ }]{e2}}
\end{aug}

\begin{abstract}
In this paper, we prove the hydrodynamic limit of the $t$-PNG model using soft techniques. One key element of the proof is the construction of a colored version of the $t$-PNG model, which allows us to apply the superadditive ergodic theorem and obtain the hydrodynamic limit, albeit without identifying the limiting constant. We then find this constant by proving a law of large numbers for the $\alpha$-points. Along the way, we construct the stationary $t$-PNG model and prove a version of Burke's theorem for it.
\end{abstract}




\end{frontmatter}

\section{Introduction}
\subsection{Background}
The distribution theory of the length of the longest increasing subsequence $\ell_n$ of a random permutation of the numbers $1, \dots, n$ with uniform measure has been under intense study in the past few decades. Now there are surveys and books \cite{aldous1999longest, groeneboom2002hydrodynamical, romik2015surprising} on this topic. 

One object of particular interest is the $n \to \infty$ asymptotic behavior of $\ell_n$. \cite{Seedlings} proved a law of large numbers $\frac{\ell_n}{\sqrt{n}} \overset{p}{\to} \gamma$ (without identifying the constant $\gamma$) by considering a Poissonized version of $\ell_n$ and relating it to a last passage percolation model called the polynuclear growth (PNG) model or Hammersley's process. Then the superadditive ergodic theorem implies the law of large numbers. There have been various approaches to determining the constant $\gamma$, see \cite{Logan-Shepp, Vershik-Kerov, seppalainen1996microscopic, aldous1999longest, groeneboom2001ulam, HammersleySourcesAndSinks}. By detailed analysis of the exact expression of the distribution function, \cite{baik1999distribution} (see also \cite{borodin2000asymptotics, johansson2001discrete})
proved a Tracy-Widom fluctuation limit theorem for $\ell_n$. 

The PNG model lies in the so-called Kardar-Parisi-Zhang universality (KPZ) class. It is natural to wonder whether there is a way to deform the PNG model so that its deformation also belongs to the KPZ universality class. \cite{aggarwal2021deformed} recently introduced a one-parameter deformation of the PNG model called the $t$-PNG model. The $t$-PNG model is also related to the length of the longest increasing subsequence from the perspective of patience sorting \cite{aldous1999longest}.
Using the method of integrable probability, the authors of \cite{aggarwal2021deformed} proved a Tracy-Widom fluctuation limit theorem for the $t$-PNG model that substantially generalizes the result of \cite{baik1999distribution}. This shows that the $t$-PNG model also belongs to the KPZ universality class.
 
Inspired by \cite{Seedlings, aldous1999longest, groeneboom2001ulam, HammersleySourcesAndSinks}, the purpose of our paper is to study the $t$-PNG model using soft arguments. 
The main result is a shape theorem for the $t$-PNG model. We will view the $t$-PNG model from two different perspectives:
as a single-colored projection of the colored $t$-PNG model and as an interacting particle system. We will see how a combination of these perspectives leads us to the main result. 

\subsection{The $t$-PNG model}
Let us proceed to define the $t$-PNG model. Fix $t \in [0,1]$. First, we place a Poisson point process with intensity $1$ on the upper-right quadrant representing \emph{nucleations}. 
We draw lines emanating from each of these nucleations in both the upward and rightward directions until they collide with one another. We call this collision point an \emph{intersection point}. Given the Poisson nucleations, we sample the outcomes of the intersection points (lines will either cross or annihilate each other) starting with the intersection point which has the smallest sum of $x$- and $y$- coordinates and moving sequentially outward. At an intersection point, the two lines will cross each other with probability $t$ and will annihilate each other with probability $1-t$, forming a corner. We call these two types of intersection points \emph{crossing points} and \emph{corner points}, respectively. Note that when 
lines cross, they might generate new intersection points. We refer to Figure \ref{fig:def-qpng} for a sampling of the $t$-PNG model. See also Appendix \ref{sec:simulations} for computer simulations of the model for different values of $t$. 

Taking $t=0$, we recover the usual PNG model. The reason that we use the parameter $t$ as the deformation parameter (following \cite{aggarwal2021deformed}) is due to the model's connection to the Hall-Littlewood polynomials, which use the parameter $t$.

\begin{figure}[t]
\centering
\begin{tikzpicture}[scale=0.6]
\draw[->] (0, 0) -- (6, 0); 
\draw[->] (0, 0) -- (0, 6);
\node at (3, 3) {$\times$};
\node at (1, 5) {$\times$};
\node at (2, 2) {$\times$};
\node at (.5, .5) {\red{$0$}};
\node at (2.5, 2.5) {\red{$1$}};
\node at (3.5, 3.5) {\red{$2$}};
\draw[fill] (6, 6) circle (0.05);
\draw[dashed] (0, 6) -- (6, 6) -- (6, 0);
\node at (6.35, 6.35) {$(x, y)$};
\draw[thick] (1, 6) -- (1, 5) -- (2, 5) -- (2, 2) -- (6, 2);
\draw[thick] (3, 6) -- (3, 3) -- (6, 3);
\draw[fill, blue] (2, 5) circle (0.1);

\begin{scope}[xshift = 10cm]
\draw[->] (0, 0) -- (6, 0); 
\draw[->] (0, 0) -- (0, 6);
\node at (3, 3) {$\times$};
\node at (1, 5) {$\times$};
\node at (2, 2) {$\times$};
\node at (.5, .5) {\red{$0$}};
\node at (2.5, 2.5) {\red{$1$}};
\node at (3.5, 3.5) {\red{$2$}};
\node at (1.5, 5.5) {\red{$1$}};

\draw[fill] (6, 6) circle (0.05);
\draw[dashed] (0, 6) -- (6, 6) -- (6, 0);
\node at (6.35, 6.35) {$(x, y)$};
\draw[thick] (1, 6) -- (1, 5) -- (2, 5) -- (2, 2) -- (6, 2);
\draw[thick] (2, 6) -- (2, 5) -- (3, 5) -- (3, 3) -- (6, 3);
\draw[fill, red] (2, 5) circle (0.1);
\draw[fill, blue] (3, 5) circle (0.1);
\end{scope}

\begin{scope}[xshift = 20cm]
\draw[->] (0, 0) -- (6, 0); 
\draw[->] (0, 0) -- (0, 6);
\node at (3, 3) {$\times$};
\node at (1, 5) {$\times$};
\node at (2, 2) {$\times$};
\node at (.5, .5) {\red{$0$}};
\node at (2.5, 2.5) {\red{$1$}};
\node at (3.5, 3.5) {\red{$2$}};
\node at (1.5, 5.5) {\red{$1$}};
\node at (2.5, 5.5) {\red{$2$}};
\node at (3.5, 5.5) {\red{$3$}};
\draw[fill] (6, 6) circle (0.05);
\draw[dashed] (0, 6) -- (6, 6) -- (6, 0);
\node at (6.35, 6.35) {$(x, y)$};
\draw[thick] (1, 6) -- (1, 5) -- (2, 5) -- (2, 2) -- (6, 2);
\draw[thick] (2, 6) -- (2, 5) -- (3, 5) -- (3, 3) -- (6, 3);
\draw[thick] (3, 6) -- (3, 5) -- (6,5) ;
\draw[fill, red] (2, 5) circle (0.1);
\draw[fill, red] (3, 5) circle (0.1);
\end{scope}

\end{tikzpicture}
\caption{Possible samplings of the $t$-PNG model. $\times$ are the Poisson nucleations. The red points are the crossing points, and the blue points are the corner points. The numbers denote the height function in each section. Left panel: With probability $1-t$, the two colliding lines form a corner point. Center panel: With probability $t(1-t)$ we have a crossing point and a corner point. Right panel: Finally, with probability $t^2$, we have two crossing points.}
\label{fig:def-qpng}
\end{figure}
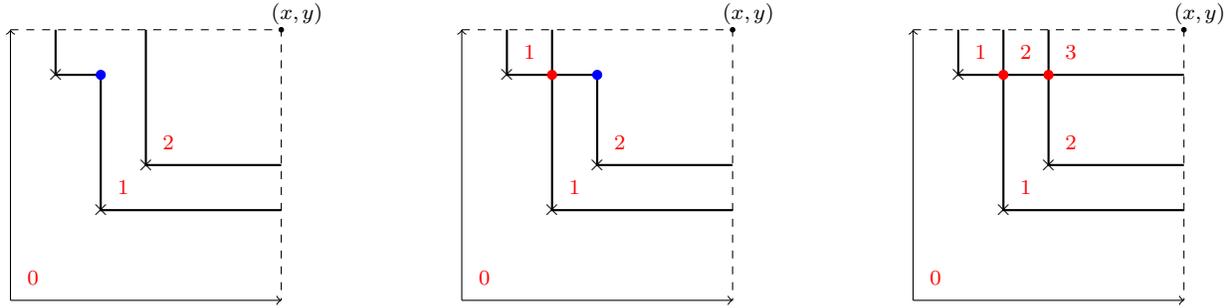

The main object we are interested in studying is the height function. 
As with the PNG model, we define the height function at the origin to be zero. Whenever we cross a line from left to right or from bottom to top, the height function increases by $1$ (see Figure \ref{fig:def-qpng}). To avoid ambiguity, we let the height function be right-continuous. We use $N(x, y)$ to denote the height function at the location $(x, y)$. Note that we hide the dependence on $t$ in the notation since it will be clear from the context. 

We can extend the above definitions to define the $t$-PNG model with boundary data. 
Fix a vector of locations on the positive $x$-axis which we will call \emph{sources} and a vector of locations on the positive $y$-axis which we will call \emph{sinks} such that on any rectangle $[0, m] \times [0,n]$ there are only finitely many sources and sinks on the bottom-left boundary. We treat the sources and sinks as additional nucleations and sample the model as before, ignoring lines that go along either the $x$-axis or $y$-axis (see Figure \ref{fig:def-qpng-sources-sinks}). The height function of the $t$-PNG model with boundary data is defined in the same way as before.

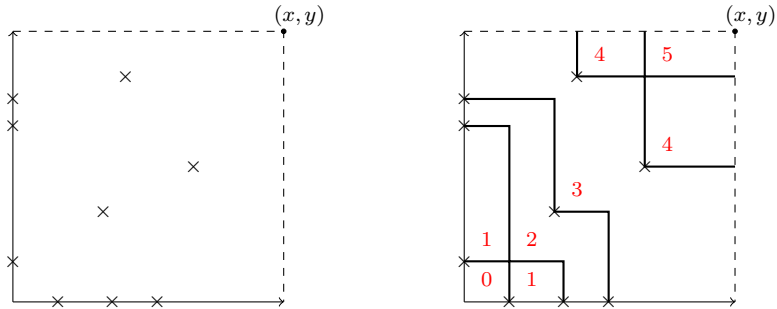
\begin{figure}[t]
\centering
\begin{tikzpicture}[scale=0.6]
\draw[->] (0, 0) -- (6, 0); 
\draw[->] (0, 0) -- (0, 6);
\node at (4, 3) {$\times$};
\node at (2.5, 5) {$\times$};
\node at (2, 2) {$\times$};

\draw[fill] (6, 6) circle (0.05);
\draw[dashed] (0, 6) -- (6, 6) -- (6, 0);
\node at (6.35, 6.35) {$(x, y)$};

\node at (1, 0) {$\times$};
\node at (2.2, 0) {$\times$};
\node at (3.2, 0) {$\times$};

\node at (0, .9) {$\times$};
\node at (0, 3.9) {$\times$};
\node at (0,4.5) {$\times$};



\begin{scope}[xshift = 10cm]
\draw[->] (0, 0) -- (6, 0); 
\draw[->] (0, 0) -- (0, 6);
\node at (4, 3) {$\times$};
\node at (2.5, 5) {$\times$};
\node at (2, 2) {$\times$};

\draw[fill] (6, 6) circle (0.05);
\draw[dashed] (0, 6) -- (6, 6) -- (6, 0);
\node at (6.35, 6.35) {$(x, y)$};

\draw[thick] (1, 0) -- (1, 0.9) -- (0, 0.9);
\draw[thick] (2.2, 0) -- (2.2, 0.9) -- (1, 0.9) --(1, 3.9) -- (0, 3.9);
\draw[thick] (3.2, 0) -- (3.2, 2) -- (2,2) --(2,4.5) -- (0, 4.5);
\draw[thick] (6,3) -- (4,3) -- (4,5) --(2.5,5)  -- (2.5,6);
\draw[thick] (6,5) -- (4,5) -- (4,6);

\node at (1, 0) {$\times$};
\node at (2.2, 0) {$\times$};
\node at (3.2, 0) {$\times$};

\node at (0, .9) {$\times$};
\node at (0, 3.9) {$\times$};
\node at (0,4.5) {$\times$};



\node at (.5, .5) {\red{$0$}};
\node at (.5, 1.4) {\red{$1$}};
\node at (1.5, .5) {\red{$1$}};
\node at (1.5, 1.4) {\red{$2$}};
\node at (2.5, 2.5) {\red{$3$}};
\node at (3, 5.5) {\red{$4$}};
\node at (4.5, 3.5) {\red{$4$}};
\node at (4.5, 5.5) {\red{$5$}};

\end{scope}

\end{tikzpicture}
\caption{A sampling of the $t$-PNG model with sources on the bottom boundary and sinks on the left boundary}
\label{fig:def-qpng-sources-sinks}
\end{figure}

\subsection{Statement of main result}
%
We present our main result, which is a shape theorem for the $t$-PNG model. 
\begin{theorem}\label{hydrodynamic_limit} 
Fix $t \in [0, 1)$. Let 
$N(x, y)$
be the height function of the $t$-PNG model as defined above. Then the following 
shape theorem holds: With probability $1$, we have for all fixed $x, y > 0$,
\begin{equation*}
\lim_{s \to \infty} \frac{N (sx, sy)}{s} = \frac{2\sqrt{xy}}{\sqrt{1-t}}. 
\end{equation*}
As a consequence, we have $\frac{N (x,y)}{\sqrt{xy}} \overset{p}{\to} \frac{2}{\sqrt{1-t}}$ as $xy \to \infty$.
\begin{rmk}
When $t = 0$, our result recovers the shape theorem for the PNG model that was proved in \cite{Logan-Shepp, Vershik-Kerov, aldous-diaconis, seppalainen1996microscopic, groeneboom2001ulam, HammersleySourcesAndSinks}. If we take $t = 1$, the right-hand side blows up. In that case,  $N(x, y)$ equals the number of Poisson nucleations in the rectangle $[0, x] \times [0, y]$. The asymptotic behavior of $N(sx, sy)$ then follows from the central limit theorem for a Poisson random variable. 
\end{rmk}
\begin{rmk}
Using the methods of integrable probability, \cite{aggarwal2021deformed} showed that as $s\to \infty$,  
\begin{equation*}
\frac{N(sx, sy) - \frac{2s\sqrt{xy}}{\sqrt{1-t}}}{(1-t)^{-\frac{1}{6}}s^{\frac{1}{3}} (xy)^{\frac{1}{6}}} \Rightarrow F_2
\end{equation*}  
where $F_2$ is the Tracy-Widom distribution. Our theorem uses a softer technique to extract the first-order asymptotic of $N(sx, sy)$ at the almost sure level.
\end{rmk}
\begin{rmk}\label{rmk:localconvergence}
A natural question to ask is whether one can prove anything about the local convergence of the model. When $t = 0$, \cite{aldous-diaconis} showed that 
\begin{equation*}
\{N(as + x, s) - N(as, s), x \in (-\infty, \infty)\} 
\end{equation*}
converges in distribution to a Poisson point process with intensity $\frac{1}{\sqrt{a}}$.
The proof in \cite{aldous-diaconis} is via identifying the PNG model as an interacting particle system on the real line and classifying its stationary distributions as convex combinations of Poisson point processes.

We believe that 
for general $t \in [0, 1)$, the same process 
converges in distribution to a Poisson point process with intensity $\frac{1}{\sqrt{(1-t)a}}$. The value $\frac{1}{\sqrt{(1-t)}a}$ comes from taking the spatial partial derivative of the hydrodynamic limit at $x = a$. 
The difficulty is that although the $t$-PNG model can still be viewed as an interacting particle system on any finite interval (see Section \ref{sec:upper-bound}), it is not clear how to extend the definition of the particle system to the entire real line due to its intricate dynamics.

\end{rmk}

\end{theorem}

A natural first step to proving the hydrodynamic limit theorem of the PNG model and other last passage percolation models is to apply the superadditive ergodic theorem. To apply the superadditive ergodic theorem (see Theorem \ref{thm:liggettergodic}), one needs to construct a family of superadditive random variables $\{X_{m, n}: 0 \leq m \leq n\}$ where $\{X_{0, n}, n \geq 0\}$ records the height function. A subfamily of the random variables also needs to be ergodic. For the PNG model, one can define $X_{m, n}$ as the length of the longest up-right path from $(m, m)$ to $(n, n)$ (we allow segments that go straight up or to the right), where the length of a path is defined as the number of nucleation points that it collects along its trajectory. 

Let us introduce a few notions for the $t$-PNG model so that we can try to modify the approach above for the case where $t>0$. We define two sets of points called $\alpha$-points and $\beta$-points, which generalize the definitions in \cite{groeneboom2001ulam}. We define the set of $\alpha$-points as the union of the Poisson nucleations and crossing points. We define the set of $\beta$-points as the union of corner points and crossing points. In particular, the intersection of the set of $\alpha$-points and the set of $\beta$-points is the set of crossing points. 

We redefine the length of an up-right path to be the number of $\alpha$-points it collects. Then the height function of the $t$-PNG model is equal to the length of the longest up-right path. Hence, one can naively try to define the random variables $\{X_{m, n}: 0 \leq m \leq n\}$ in a similar way as above, but with our new definition of length. With this definition, the $X_{m,n}$ are superadditive as desired. However, they are no longer stationary since the number of crossing points in a box no longer just depends on the number of Poisson nucleations in that box---it also depends on the lines entering that box from the left and bottom boundaries. Therefore, we no longer even have stationarity and thus no ergodicity.

It turns out that there is indeed a way construct a family of random variables $\{X_{m, n}: 0 \leq m \leq n\}$ that satisfies the conditions required by the superadditive ergodic theorem. The family $\{X_{m, n}: 0 \leq m \leq n\}$ comes from a colored version of the $t$-PNG model that we are going to introduce. 
\subsection{The colored $t$-PNG model}

As we see from above, for the $t$-PNG model, two lines emanate rightward and upward from each Poisson nucleation point. The key rule for sampling the $t$-PNG model after fixing the Poisson nucleation points is that when two lines meet, they cross with probability $t$ and annihilate each other with probability $1-t$. 
This rule can be encoded into a stochastic matrix. 
We associate each intersection point with a 4-tuple $i, j, k, l \in \{0, 1\}$ which specifies the number of lines on the bottom, left, top, and right, respectively. We define a stochastic matrix $\LL^1$ as follows.  
\begin{definition}\label{def:Lmatrix}
The matrix $\LL^1$ is indexed by a 4-tuple $i, j, k, l \in \{0, 1\}$, where $i, j, k, l$ denote the number of lines (either zero or one) on the bottom, left, top, and right of an intersection point. We define
\begin{align*}
&\LL^1(1, 1; 1, 1) = t,\qquad \LL^1(1, 1; 0, 0) =1-t, \qquad \LL^1(1, 0; 1, 0) = 1,\\
&\LL^1(0, 1; 0, 1) = 1, \qquad \LL^1(0, 0; 0, 0) = 1.
\end{align*}
For all other $i, j, k, l \in \{0, 1\}$, we set $\LL^1(i, j; k, l) = 0$. For fixed input lines $i, j \in \{0, 1\}$, $\LL^1(i, j; \cdot, \cdot)$ is a probability measure on the output lines. See Figure \ref{fig:L-matrix} for illustration. 
\end{definition}
\begin{figure}[ht]
	\centering
	\begin{tikzpicture}
	\begin{scope}[xshift = 2cm]
	\draw[thick] (0, 0) -- (0.5, 0) -- (0.5, -0.5);
	\node at (0.5, -1) {$1-t$};
	\end{scope}
	\begin{scope}[xshift = 0cm]
	\draw[thick] (0, 0) -- (1, 0); 
	\draw[thick] (0.5, -0.5) -- (0.5, 0.5);
	\node at (0.5, -1) {$t$};
	\end{scope}
	\begin{scope}[xshift = 6cm]
	\draw[thick] (0, 0) -- (1, 0); 
	\node at (0.5, -1) {$1$};
	\end{scope}
	\begin{scope}[xshift = 4cm]
	\draw[thick] (0.5, -0.5) -- (0.5, 0.5);
	\node at (0.5, -1) {$1$};
	\end{scope}
	\begin{scope}[xshift = 8cm]
	\node at (0.5, -1) {$1$};
	\end{scope}
	\end{tikzpicture}
	\caption{We draw all the intersection configurations that have non-zero weights.}
	\label{fig:L-matrix}
\end{figure}
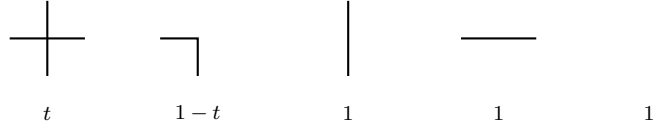
In Figure \ref{fig:L-matrix}, the first two configurations represent a vertical line intersecting with a horizontal line. The two lines cross with probability $t$ and annihilate each other with probability $1-t$. The third and fourth configurations depict that a horizontal (resp. vertical) line will continue as long as it does not meet a vertical (resp. horizontal) line. The last empty configuration represents the fact that lines cannot emerge out of nowhere (except at a Poisson nucleation, but those have already been fixed before sampling the rest of the model).  

For the colored $t$-PNG model, we denote the different colors by integers $i \in \mathbb{N}$. We allow multiple (but only finitely many) lines with different colors to travel together.  We say that the color $i$ has higher priority than the color $j$ if $i < j$. The only restriction is that lines traveling together must have different colors. 
The colored $t$-PNG model is defined by specifying the sampling rule for when horizontal lines and vertical lines meet. The sampling rule is given by a family of stochastic matrices $\{\mathsf{L}^n, n \in \mathbb{N}\}$ that are consistent. More concretely, the matrix $\LL^n$ has both rows and columns indexed by $\{0, 1\}^n \times \{0, 1\}^n$. The  matrix elements are given by $\mathsf{L}^n (\bi, \bj; \bk, \bl)$, where the four vectors $\bi, \bj, \bk, \bl \in \{0, 1\}^n$ specify the number of lines (either zero or one) of each color in $\{1, \dots, n\}$ on the bottom, left, top, and right of an intersection, respectively (see Figure \ref{fig:intersectionconfig}). The stochastic matrices give a probability measure on the output lines $\bk, \bl$ from an intersection point given the input lines $\bi, \bj$. 
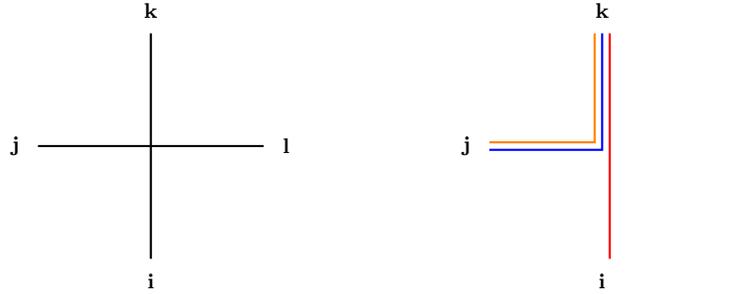
\begin{figure}[ht]
	\centering
	\begin{tikzpicture}
	\begin{scope}
	\draw[thick] (0, 0) -- (3, 0);
	\draw[thick] (1.5, -1.5) -- (1.5, 1.5);
	\node at (1.5, -1.8) {$\bi$}; 
	\node at (-0.3, 0) {$\bj$}; 
	\node at (1.5, 1.8) {$\bk$};
	\node at (3.3, 0) {$\bl$};
	\end{scope}
	\begin{scope}[xshift = 6cm]
	\draw[thick, blue] (0, -0.05) -- (1.5, -0.05) -- (1.5, 1.5);
	\draw[thick, orange] (0, 0.05) -- (1.4, 0.05) -- (1.4, 1.5);
	\draw[thick, red] (1.6, -1.5) -- (1.6, 1.5);
	\node at (1.5, -1.8) {$\bi$}; 
	\node at (-0.3, 0) {$\bj$}; 
	\node at (1.5, 1.8) {$\bk$};
	\node at (3.3, 0) {$\bl$};
	\end{scope}
	\end{tikzpicture}
	\caption{Left panel: Fix $n \in \mathbb{N}$. At an intersection point, we have lines with colors in $\{1, \dots, n\}$ in each of the four directions, but for any given direction there can be at most one line per color. Let $\bi, \bj, \bk, \bl \in \{0, 1\}^n$ denote the number of lines on the bottom, left, top and right directions, respectively, where the $m$-th coordinate of each vector records the number of lines with color $m$. Right panel: Take $n = 3$. Let red, blue, and orange denote the colors $1$, $2$, and $3$. We illustrate an example of the configuration with  $\bi = (1, 0, 0)$, $\bj = (0, 1, 1)$,  $\bk = (1, 1, 1)$, and $\bl = (0, 0, 0)$.}
	\label{fig:intersectionconfig}
\end{figure}
 

We proceed to give a closed form to the matrices $\{\LL^n: n \geq 1\}$. We first need to introduce some notation.
\begin{definition}[$r$-fold Projection]
    For $\mathbf{x} = (x_1, \dots, x_n) \in \{0, 1\}^n$ and $r \in \{1, \dots, n\}$, we define 
	$$\mfs_r (\mathbf{x}) = \Big(\sum_{m=1}^r x_m\Big) \mod 2.$$
Then for a vertex configuration given by the 4-tuple ($\bi ,\bj; \bk, \bl$), we can define its \textit{$r$-fold projection} to be the single-colored configuration $(i, j; k,l) = (\mfs_r (\bi), \mfs_r (\bj); \mfs_r (\bk), \mfs_r (\bl))$. 
\end{definition}

In other words, we are projecting the first $r$ colors of $(\bi, \bj; \bk, \bl)$ down to a single color and ignoring all colors greater than $r$. The projection $\mathfrak{s}_r (\mathbf{x})$ corresponds to replacing the colors $1, 2, \dots, r$ with a single color and then erasing every pair of lines. Then $\mfs_r (\mathbf{x})$ denotes the number of lines that remain. This is also equivalent to replacing the total number of lines by itself mod $2$. 

As an example consider the following three-colored configuration. We adopt the convention that red, blue, and orange denote the colors $1$, $2$, and $3$ respectively.

\begin{figure}[ht]
	\centering
	\begin{tikzpicture}[scale = 1]
	\draw[red, thick] (-0.1, -0.5) -- (-0.1, 0.5);
            \draw[blue, thick] (0, -0.5) -- (0, 0.5); 
            \draw[orange, thick] (0.1, -0.5) -- (0.1, -0.2) -- (0.5, -0.2); 
            \draw[blue, thick] (-0.5, -0.1) -- (0.5, -0.1);

	\end{tikzpicture}

\end{figure}
\noindent Its $1$-fold, $2$-fold, and $3$-fold projections are given by the following table:
\begin{figure}[ht]
	\centering
	\begin{tabular}{|c|c|c|c|}
	\hline
	Procedure &$1$-fold projection & 	$2$-fold projection &	$3$-fold projection \\
	\hline 
	\begin{tikzpicture}[scale = 1.2]
	\draw[fill][white] (0.5, 0) circle (0.05);
	\draw[thick][white] (0, 0) -- (1,0);
	\draw[thick][white] (0.5, -0.5) -- (0.5,0.5);
	\node at (0.5, 0) {Step 1: Consider the first $r$ colors.};
	\end{tikzpicture}
	 &
	
	\begin{tikzpicture}[scale=1]
	\draw[thick][white] (-0.5, 0) -- (0.5,0);
	\draw[thick][white] (0, -0.5) -- (0,0.5);
	\draw[red, thick] (0, -0.5) -- (0, 0.5);
	\end{tikzpicture}
	 & 
	\begin{tikzpicture}[scale=1]
	\draw[thick][white] (-0.5, 0) -- (0.5,0);
	\draw[thick][white] (0, -0.5) -- (0,0.5);
	\draw[red, thick] (-0.1, -0.5) -- (-0.1, 0.5);
    \draw[blue, thick] (0, -0.5) -- (0, 0.5); 
    \draw[blue, thick] (-0.5, -0.1) -- (0.5, -0.1);
	\end{tikzpicture}
&

    \begin{tikzpicture}
    \draw[thick][white] (-0.5, 0) -- (0.5,0);
	\draw[thick][white] (0, -0.5) -- (0,0.5);
    \draw[red, thick] (-0.1, -0.5) -- (-0.1, 0.5);
    \draw[blue, thick] (0, -0.5) -- (0, 0.5); 
    \draw[orange, thick] (0.1, -0.5) -- (0.1, -0.2) -- (0.5, -0.2); 
    \draw[blue, thick] (-0.5, -0.1) -- (0.5, -0.1);
\end{tikzpicture}\\ 		
	\hline
	\begin{tikzpicture}[scale = 1.2]
	\draw[fill][white] (0.5, 0) circle (0.05);
	\draw[thick][white] (0, 0) -- (1,0);
	\draw[thick][white] (0.5, -0.5) -- (0.5,0.5);
	\node at (0.5, 0) {Step 2: Recolor everything black.};
	\end{tikzpicture}
    &
    \begin{tikzpicture}
    \draw[thick][white] (-0.5, 0) -- (0.5,0);
	\draw[thick][white] (0, -0.5) -- (0,0.5);
	\draw[black, thick] (0, -0.5) -- (0, 0.5);
	 \end{tikzpicture} &
	 
	\begin{tikzpicture}
	\draw[thick][white] (-0.5, 0) -- (0.5,0);
	\draw[thick][white] (0, -0.5) -- (0,0.5);	
	 \draw[black, thick] (-0.1, -0.5) -- (-0.1, 0.5);
    \draw[black, thick] (0, -0.5) -- (0, 0.5); 
    \draw[black, thick] (-0.5, -0.1) -- (0.5, -0.1);
\end{tikzpicture}

&
    \begin{tikzpicture}
	\draw[thick][white] (-0.5, 0) -- (0.5,0);
	\draw[thick][white] (0, -0.5) -- (0,0.5);
    \draw[black, thick] (-0.1, -0.5) -- (-0.1, 0.5);
    \draw[black, thick] (0, -0.5) -- (0, 0.5); 
    \draw[black, thick] (0.1, -0.5) -- (0.1, -0.2) -- (0.5, -0.2); 
    \draw[black, thick] (-0.5, -0.1) -- (0.5, -0.1);
\end{tikzpicture} \\

	\hline
	\begin{tikzpicture}[scale = 1.2]
	\draw[fill][white] (0.5, 0) circle (0.05);
	\draw[thick][white] (0, 0) -- (1,0);
	\draw[thick][white] (0.5, -0.5) -- (0.5,0.5);
	\node at (0.5, 0) {Step 3: Replace the number of lines with the number of lines modulo 2.};
	\end{tikzpicture}	
	 &
	\begin{tikzpicture}
	\draw[thick][white] (-0.5, 0) -- (0.5,0);
	\draw[thick][white] (0, -0.5) -- (0,0.5);
    \draw[black, thick] (0, -0.5) -- (0, 0.5);
    \end{tikzpicture}
    &
    \begin{tikzpicture}
	\draw[thick][white] (-0.5, 0) -- (0.5,0);
	\draw[thick][white] (0, -0.5) -- (0,0.5);
    \draw[black, thick] (-0.5, -0.1) -- (0.5, -0.1);
\end{tikzpicture}

&

    \begin{tikzpicture}
	\draw[thick][white] (-0.5, 0) -- (0.5,0);
	\draw[thick][white] (0, -0.5) -- (0,0.5);
	\draw[black, thick] (-0.5, 0)--(0,0) --(0, -.5);
\end{tikzpicture}\\
	\hline
						
\end{tabular}
\caption{\textbf{$\mathbf{r}$-fold Projections}: An example of how to take the $r$-fold projections of a given configuration. We recolor the first $r$ colors black and delete all other colors. We then replace the number of lines with the number of lines modulo 2.}
	\label{fig:mod2erasure}
\end{figure}

To compute the weight of an $n$-colored configuration $(\bi, \bj; \bk, \bl)$, we will look at the weights of all of its $r$-fold projections for $r \in \{1, \dots, n\}$, which can be computed using the matrix $\LL^1$ that we have already defined in Definition \ref{def:Lmatrix}. We want the $r$-fold projections of the colored model to have the same distribution as the single-colored $t$-PNG model. Therefore, if any of the $r$-fold projections have weight zero then we need the weight of $(\bi, \bj; \bk, \bl)$ to be zero as well, as such a configuration should not be allowed.

Furthermore, we want to disallow the case where one of the $r$-fold projections has weight $t$ and another one has weight $1-t$ since this would prevent the colored model from being superadditive. The reason for this is as follows: Superadditivity will follow from a certain monotonicity property of the height function (See Property \hyperlink{hyp:property3}{3} below). Namely, the height function of the $2$-fold projection should be greater than or equal to the height function of the $1$-fold projection.

If we allow $r$-fold projections with weights $t$ and $1-t$, then we will allow the following sampling in Figure \ref{fig:badConfig} whose $2$-fold projection has a smaller height function than its $1$-fold projection:

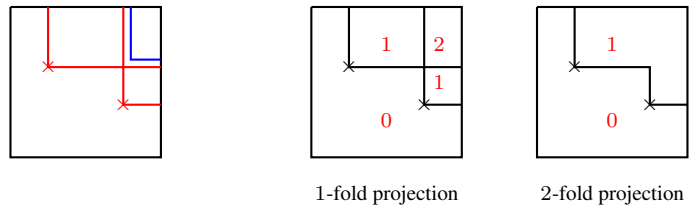
\begin{figure}[ht]
\centering
\begin{tikzpicture} 
\begin{scope}[xshift = 3cm]
\draw[thick] (0, 2) -- (0, 0) -- (2, 0) -- (2, 2) -- (0, 2);
\node at (0.5, 1.2) {$\red{\times}$};
\node at (1.5, 0.7) {$\red{\times}$};
\draw[thick, red] (0.5, 2) -- (0.5, 1.2) -- (1.5, 1.2) -- (1.5, 0.7) -- (2, 0.7);
\draw[thick, red] (1.5, 2) -- (1.5, 1.2) -- (2, 1.2);
\draw[thick, blue] (1.6, 2) -- (1.6, 1.3) -- (2, 1.3);
\end{scope}    

\begin{scope}[xshift =7cm]
\draw[thick] (0, 2) -- (0, 0) -- (2, 0) -- (2, 2) -- (0, 2);
\node at (0.5, 1.2) {$\times$};
\node at (1.5, 0.7) {$\times$};
\draw[thick] (0.5, 2) -- (0.5, 1.2) -- (1.5, 1.2) -- (1.5, 0.7) -- (2, 0.7);
\draw[thick] (1.5, 2) -- (1.5, 1.2) -- (2, 1.2);
\node at (1, -0.5) {$1$-fold projection};
\node at (1, 0.5) {$\red{0}$};
\node at (1, 1.5) {$\red{1}$};
\node at (1.7, 1) {$\red{1}$};
\node at (1.7, 1.5) {$\red{2}$};
\end{scope} 
\begin{scope}[xshift = 10cm]
\draw[thick] (0, 2) -- (0, 0) -- (2, 0) -- (2, 2) -- (0, 2);
\node at (0.5, 1.2) {$\times$};
\node at (1.5, 0.7) {$\times$};
\draw[thick] (0.5, 2) -- (0.5, 1.2) -- (1.5, 1.2) -- (1.5, 0.7) -- (2, 0.7);
\node at (1, -0.5) {$2$-fold projection};
\node at (1, 0.5) {$\red{0}$};
\node at (1, 1.5) {$\red{1}$};
\end{scope} 
\end{tikzpicture}
\caption{The weight of the $1$-fold projection is $t$ while the weight of the $2$-fold projection is $1-t$. Notice how the height function in the top-right corner of the $2$-fold projection is smaller than the height function in the top-right corner of the $1$-fold projection.}
\label{fig:badConfig}
\end{figure}

Therefore, we will define the weight of $(\bi, \bj; \bk, \bl)$ to be the minimum of the weights of its $r$-fold projections; however, instead of using the classical minimum function, we use a modified version defined as follows:

Fix $t \in [0, 1)$. Let $\mmin$ be a modification of the $\min$ function such that for $x_1, \dots, x_n \in \{0, t,1-t, 1\},$  
\begin{equation*}
\mmin\big(x_1, \dots, x_n\big) = 
\begin{cases}
0 &\qquad \text{if $x_i = t$ and $x_j = 1-t$ for some $i,j \in \{1,\dots,n\}$,}\\
\min\big(x_1, \dots, x_n\big) &\qquad \text{else.}\\
\end{cases}
\end{equation*}
For example, we have $\mmin(t, 1) = t$, $\mmin(1-t, 1) =1-t$ and $\mmin(t,1-t) = 0$. Note that we are treating $t$ and $1-t$ as indeterminates, so we can ignore the case where $t= 1-t$. With this definition, a configuration will only have nonzero weight if all of its $r$-fold projections have nonzero weight and if $t$ and $1-t$ are not both weights of different $r$-fold projections.

\begin{definition}\label{def:nLmatrix}
	Fix arbitrary $n \in \Z_{\geq 1}$. 
	We define the matrix $\LL^n$ via 
	\begin{equation}\label{eq:weight}
	\LL^n(\bi, \bj; \bk, \bl) = \mmin_{r \in \{1, \dots, n\}} \Big(\LL^1 (\mfs_r (\bi), \mfs_r (\bj); \mfs_r (\bk), \mfs_r (\bl))\Big).
	\end{equation}
\end{definition}
The table in Figure \ref{fig:computing_weights} illustrates how to compute the weights of some three-colored configurations. For the following configurations, we use the colors red, blue, and orange to represent the colors $1$, $2$, and $3$ respectively. We also draw all of the two-colored configurations with non-zero weights in Appendix \ref{sec:twocolorfig} for further illustration.

\begin{figure}[ht]
	\centering
		\begin{tabular}{|c|c|c|c|c|}
			\hline
			Configuration & $1$-fold projection  & $2$-fold projection  & $3$-fold projection & Total weight\\
			\hline
			\begin{tikzpicture}[scale = 1.2]
			\draw[thick][white] (-0.5, 0) -- (0.5,0);
	        \draw[thick][white] (0, -0.6) -- (0,0.6);
	        \draw[red, thick] (-0.1, -0.5) -- (-0.1, 0.5);
            \draw[blue, thick] (0, -0.5) -- (0, 0.5); 
            \draw[orange, thick] (0.1, -0.5) -- (0.1, -0.2) -- (0.5, -0.2); 
            \draw[blue, thick] (-0.5, -0.1) -- (0.5, -0.1);
            \node at (0, -0.8) {};
			\end{tikzpicture}
			&
			\begin{tikzpicture}[scale = 1.2]
			\draw[thick][white] (-0.5, 0) -- (0.5,0);
	        \draw[thick][white] (0, -0.6) -- (0,0.6);
	        \draw[black, thick] (0, -0.5) -- (0, 0.5);
            \node at (0, -0.8) {$1$ };
			\end{tikzpicture}
			&
			\begin{tikzpicture}[scale = 1.2]
			\draw[thick][white] (-0.5, 0) -- (0.5,0);
	        \draw[thick][white] (0, -0.6) -- (0,0.6);
            \draw[black, thick] (-0.5, -0.1) -- (0.5, -0.1);
            \node at (0, -0.8) {$1$};
			\end{tikzpicture}
			&
			\begin{tikzpicture}[scale = 1.2]
	        \draw[thick][white] (-0.5, 0) -- (0.5,0);
	        \draw[thick][white] (0, -0.6) -- (0,0.6);
            \draw[black, thick] (-0.5, -0.1) -- (0, -0.1)--(0,-0.5);
            \node at (0, -0.8) {$1 - t$};
			\end{tikzpicture}
			&
			\begin{tikzpicture}[scale = 1.2]
			\draw[fill][white] (0.5, 0) circle (0.05);
        	\draw[thick][white] (0, 0) -- (1,0);
        	\draw[thick][white] (0.5, -0.5) -- (0.5,0.5);
        	\node at (0.5, 0) {$\mmin(1,1,1-t) = 1-t$};
			\end{tikzpicture}
           	
			\\
			\hline
    		\begin{tikzpicture}[scale = 1.2]
    	    \draw[thick][white] (-0.5, 0) -- (0.5,0);
    	   \draw[thick][white] (0, -0.6) -- (0,0.6);
            \draw[red, thick]
            (-0.5, -0.1) -- (0.5, -0.1);
            \draw[red, thick](0, -0.5) -- (0, 0.5); 
            \draw[blue, thick] (-0.5, 0)--(0.5, 0);
            \draw[orange, thick] (-0.5, 0.1)--(-0.1, 0.1)--(-0.1, 0.5);
            \node at (0, -0.8) { };
    		\end{tikzpicture}
			&
    		\begin{tikzpicture}[scale = 1.2]
    	    \draw[thick][white] (-0.5, 0) -- (0.5,0);
    	   \draw[thick][white] (0, -0.6) -- (0,0.6);
    	   \draw[black, thick]
            (-0.5, -0) -- (0.5, -0);
            \draw[black, thick](0, -0.5) -- (0, 0.5); 
            \node at (0, -0.8) {$t$ };
    		\end{tikzpicture}
			&
    		\begin{tikzpicture}[scale = 1.2]
     	    \draw[thick][white] (-0.5, 0) -- (0.5,0);
    	    \draw[thick][white] (0, -0.6) -- (0,0.6);\draw[black, thick](0, -0.5) -- (0, 0.5);
            \node at (0, -0.8) {$1$};
    		\end{tikzpicture}
			&
    		\begin{tikzpicture}[scale = 1.2]
    		\draw[thick][white] (-0.5, 0) -- (0.5,0);
    	   \draw[thick][white] (0, -0.6) -- (0,0.6);
            \draw[black, thick]
            (-0.5, -0) -- (0, -0);
            \draw[black, thick](0, -0.5) -- (0, 0); 
            \node at (0, -0.8) {$1-t$ };
		    \end{tikzpicture}
			&
        	\begin{tikzpicture}[scale = 1.2]
        	\draw[fill][white] (0.5, 0) circle (0.05);
        	\draw[thick][white] (0, 0) -- (1,0);
        	\draw[thick][white] (0.5, -0.5) -- (0.5,0.5);
        	\node at (0.5, 0) {$\mmin(t,1,1-t) = 0$};
        	\end{tikzpicture}	
                    
			\\
			\hline
			
\end{tabular}
\caption{Examples of how to compute the weights of three-colored configurations} 
	\label{fig:computing_weights}
\end{figure}

It is not a priori clear that $\LL^n$ is stochastic; this will be proved in Section \ref{sec:multicolor}. The family of stochastic matrices $\{\LL^n,  n \geq 1\}$ satisfies the following three properties: 

\hypertarget{hyp:property1}{\textbf{Property 1}}
(Color Ignorance): Lines with higher priority colors ignore lines with lower priority colors (see Figure \ref{fig:properties1-2} for illustration). For instance, the lines with colors belonging to $\{1, \dots, m\}$ ignore the behavior of lines with colors greater than $m$. This means that if we sample the $n$-colored model and ignore the lines with color greater than $m$, the remaining lines will reduce to the $m$-colored model. In particular, if we ignore the lines of color $2$ in the two-colored model, then the lines of color $1$ have the same distribution as the single-colored $t$-PNG model. On the other hand, if we ignore the lines of color $1$, the lines of color $2$ do \textit{not} have the same distribution as the single-colored $t$-PNG model. Because of this property, we will be able to define the $X_{m,n}$ in a way that maintains ergodicity, as we will see in Section \ref{sec:subadditivity}.  
	
\hypertarget{hyp:property2}{\textbf{Property 2}} (Mod 2 Erasure): Fix arbitrary integers $1 \leq r_1 < \dots < r_m \leq n$. We can project the matrix $\LL^n$ to 
	$\LL^m$ if we replace the colors in $\{r_{k-1}+1, \dots, r_k\}$ with color $k$ for each $k \in\{1, \dots, m\}$ and then erase every pair of lines that has the same color (see Figure \ref{fig:properties1-2} for illustration). Because of this property, we can project the colored model down to the single-colored $t$-PNG model. This will ensure that the random variables $X_{0,n}$ will record the height function of the $t$-PNG model, which is the quantity that we are interested in studying.


\hypertarget{hyp:property3}{\textbf{Property 3}} (Monotonicity of the Height Function): 
Suppose we have a sampling of the two-colored $t$-PNG configuration on a rectangle $[0,x] \times [0,y]$ where all nucleations are of the first color, and all sources and sinks are of the second color. Let $N^1(x,y)$ denote the height function at $(x,y)$ of the $1$-fold projection of all of the lines in the rectangle. Let $N^2(x,y)$ denote the height function at $(x,y)$ of the $2$-fold projection of the lines. Then $$N^1(x,y) \leq N^2(x,y).$$ In other words, adding a second color to the model does not decrease the height function. This property will be crucial to proving the superadditivity of the random variables $X_{m,n}$.

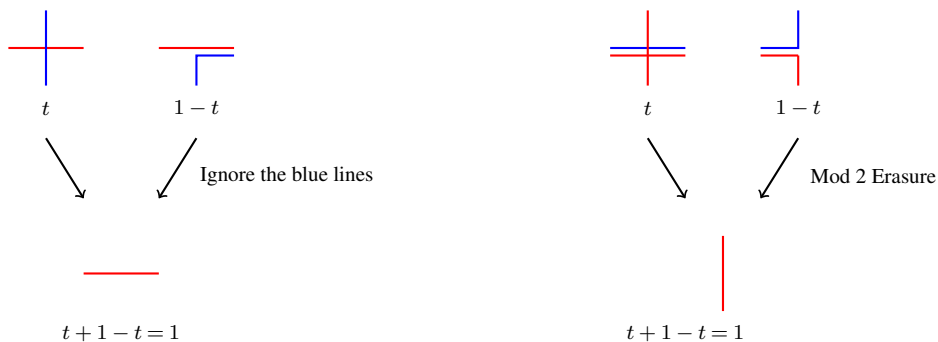
\begin{figure}[t]
	\centering
	\begin{tikzpicture}
	\begin{scope}[xshift = 0cm]
	\draw[red, thick] (0, 0) -- (1, 0);
	\draw[blue, thick] (0.5, -0.5) -- (0.5, -0.1) -- (0.5, 0.5);
	\node at (0.5, -0.8) {$t$};
	\draw[black, thick, ->] (0.5, -1.2) -- (1, -2);
	\end{scope}
	\begin{scope}[xshift = 2cm]
	\draw[red, thick] (0, 0) -- (1, 0);
	\draw[blue, thick] (0.5, -0.5) -- (0.5, -0.1) -- (1, -0.1);
	\node at (0.5, -0.8) {$1-t$};
	\draw[black, thick, ->] (0.5, -1.2) -- (0, -2);
	\node at (1.7, -1.7) {Ignore the blue lines};
	\end{scope}
	\begin{scope}[yshift = -3cm, xshift = 1cm]
	\draw[red, thick] (0, 0) -- (1, 0);
	\node at (0.5, -0.8) {$t +1-t = 1$};
	\end{scope}
	
	\begin{scope}[xshift = 2cm]
	\begin{scope}[xshift = 6cm]
	\draw[blue, thick] (0, 0) -- (1, 0);
	\draw[red, thick] (0, -0.1) -- (1, -0.1);
	\draw[red, thick] (0.5, -0.5) -- (0.5, 0.5);
	\node at (0.5, -0.8) {$t$};
	\draw[black, thick, ->] (0.5, -1.2) -- (1, -2);
	\end{scope}
	\begin{scope}[xshift = 8cm]
\draw[blue, thick] (0, 0) -- (0.5, 0) -- (0.5, 0.5);
	\draw[red, thick] (0, -0.1) -- (0.5, -0.1);
	\draw[red, thick] (0.5, -0.5) -- (0.5, -0.1);
	\node at (0.5, -0.8) {$1-t$};
	\draw[black, thick, ->] (0.5, -1.2) -- (0, -2);
	\node at (1.5, -1.7) {Mod 2 Erasure};
	\end{scope}
	\begin{scope}[yshift = -3cm, xshift = 7cm]
	\draw[red, thick] (0.5, -0.5) -- (0.5, 0.5);
	\node at (0, -0.8) {$t +1-t = 1$};
	\end{scope}
	\end{scope}
	\end{tikzpicture}
	\caption{Left Panel: Color Ignorance. The top two configurations have the same input lines, and if we ignore the blue lines then they both equal the same single-colored configuration---a horizontal line. In fact, these are the only two-colored configurations with the given input lines whose first color is a horizontal line. Therefore the sum of their weights equals the weight of the horizontal line. Right Panel: Mod 2 Erasure. The top two configurations are the only two-colored configurations with the given input lines whose $2$-fold projection is a vertical line. Therefore the sum of their weights equals the weight of the horizontal line.}
	\label{fig:properties1-2}
\end{figure}

The proofs of the first two properties are in Section \ref{sec:multicolor}. The third property follows from Lemma \ref{lem:attractivity}. Note that if we take $t=0$ then the two-colored $t$-PNG model degenerates to the two-colored PNG model defined in \cite{HammersleySourcesAndSinks}. To our best knowledge, it seems that the sampling rule of the colored $t$-PNG model has not been defined earlier.

In \cite{HammersleySourcesAndSinks}, the authors take the viewpoint of interacting particle systems. The lines of the second color for the PNG model are the trajectory lines of second class particles for the Hammersley's process. This was one motivation for arriving at our definition of the colored model. We define our colored model so that the lines of the second color are the second class particles for the $t$-Hammersley process that we introduce in Section \ref{sec:upper-bound}. A second motivation for our definition comes from the colored stochastic six vertex model as detailed in the next subsection.

\subsubsection{Connection to the stochastic six vertex model} 
The stochastic six vertex (S6V) model is a classical model in two-dimensional statistical physics. The model was introduced in \cite{gwa1992six} as a special case of the six vertex model \cite{lieb1967residual, baxter2016exactly}. We associate six possible configurations to each vertex in $\mathbb{Z}_{\geq 0}^2$ as illustrated in Figure \ref{fig:vertexconfig}. The weight of each configuration is parameterized by two parameters $b_1, b_2 \in [0, 1]$. We view the lines entering the vertex from the left and the bottom as input lines and view the lines leaving to the right and above as output lines. The S6V model is stochastic since if we are given the number of input lines from the left and bottom, the sum of the weights of all possible configurations with that input equals $1$.

We view the S6V model as a stochastic path ensemble on $\mathbb{Z}_{\geq 0}^2$.
We fix boundary conditions on the axis $\mathbb{Z}_{\geq 0} \times \{0\}$ (resp. $\{0\} \times \mathbb{Z}_{\geq 0}$) which indicate whether there is an input line entering each vertex along the axis from the bottom (resp. left). Starting from the vertex $(0, 0)$, we tile the given site with one of the six vertex configurations where we only consider configurations whose input lines match the input lines of the given vertex. We then assign an allowed configuration with probability given by the weight of the configuration. This tiling construction then progresses sequentially in the linear order $(0, 0),(1, 0),(0, 1),(1, 1),(2, 0),(1, 1), (0, 2), \dots$ to the entire quadrant (see the left panel of Figure \ref{fig:sampling}). 

\begin{figure}[ht]
	\centering
		\begin{tabular}{|c|c|c|c|c|c|c|}
			\hline
			Type & I & II & III & IV & V & VI \\
			\hline
			\begin{tikzpicture}[scale = 1.5]
			\draw[fill][white] (0.5, 0) circle (0.05);
			\draw[thick][white] (0, 0) -- (1,0);
			\draw[thick][white] (0.5, -0.5) -- 
			(0.5,0.5);
			\node at (0.5, 0) {Configuration};
			\end{tikzpicture}
			&
			\begin{tikzpicture}[scale = 1.2]
			\draw[fill] (0.5, 0) circle (0.05);
			\draw[thick] (0, 0) -- (1,0);
			\draw[thick] (0.5, -0.5) -- (0.5,0.5);
			\end{tikzpicture}
			&
			\begin{tikzpicture}[scale = 1.2]
			\draw[thick][white] (0, 0) -- (1,0);
			\draw[thick][white] (0.5, -0.5) -- (0.5,0.5);
			\draw[fill] (0.5, 0) circle (0.05);
			\end{tikzpicture}
			&
			\begin{tikzpicture}[scale = 1.2]
			\draw[thick][white] (0, 0) -- (1,0);
			\draw[thick] (0.5, -0.5) -- (0.5,0.5);
			\draw[fill] (0.5, 0) circle (0.05);
			\end{tikzpicture}
			&
			\begin{tikzpicture}[scale = 1.2]
			\draw[thick][white] (0, 0) -- (0.5,0);
			\draw[thick][white] (0.5, 0) -- (0.5, 0.5);
			\draw[thick] (0.5, 0) -- (1, 0);
			\draw[thick] (0.5, -0.5) -- (0.5, 0);
			(0.5,0.5);
			\draw[fill] (0.5, 0) circle (0.05);
			\end{tikzpicture}
			&
			\begin{tikzpicture}[scale = 1.2]
			\draw[thick] (0, 0) -- (1,0);
			\draw[thick][white] (0.5, -0.5) -- (0.5,0.5);
			\draw[fill] (0.5, 0) circle (0.05);
			\end{tikzpicture}
			&
			\begin{tikzpicture}[scale = 1.2]
			\draw[thick] (0, 0) -- (0.5,0);
			\draw[thick] (0.5, 0) -- (0.5, 0.5);
			\draw[thick][white] (0.5, 0) -- (1, 0);
			\draw[thick][white] (0.5, -0.5) -- (0.5, 0);
			(0.5,0.5);
			\draw[fill] (0.5, 0) circle (0.05);
			\end{tikzpicture}
			\\
			\hline
			Weight 
			& 1 & 1 & $b_1$ & $1- b_1$ & $b_2$ & $1-b_2$\\
			\hline
\end{tabular}
\caption{Six types of configurations for the S6V model.}
	\label{fig:vertexconfig}
\end{figure}

To relate this to the $t$-PNG model, we horizontally complement the S6V model. In other words, if there is a horizontal line, we erase it; if there is no horizontal line, we add it (see the right panel of Figure \ref{fig:sampling}).  In Figure \ref{fig:cs6v}, we show the vertex configurations after horizontal complementation.
\begin{figure}[ht]
	\centering
		\begin{tabular}{|c|c|c|c|c|c|c|}
			\hline
			Type & I & II & III & IV & V & VI \\
			\hline
			\begin{tikzpicture}[scale = 1.5]
			\draw[fill][white] (0.5, 0) circle (0.05);
			\draw[thick][white] (0, 0) -- (1,0);
			\draw[thick][white] (0.5, -0.5) -- 
			(0.5,0.5);
			\node at (0.5, 0) {Configuration};
			\end{tikzpicture}
			&
			\begin{tikzpicture}[scale = 1.2]
			\draw[thick, white] (0, 0) -- (1,0);
			\draw[thick] (0.5, -0.5) -- (0.5,0.5);
			\draw[fill] (0.5, 0) circle (0.05);
			\end{tikzpicture}
			&
			\begin{tikzpicture}[scale = 1.2]
			\draw[thick] (0, 0) -- (1,0);
			\draw[thick][white] (0.5, -0.5) -- (0.5,0.5);
			\draw[fill] (0.5, 0) circle (0.05);
			\end{tikzpicture}
			&
			\begin{tikzpicture}[scale = 1.2]
			\draw[thick] (0, 0) -- (1,0);
			\draw[thick] (0.5, -0.5) -- (0.5,0.5);
			\draw[fill] (0.5, 0) circle (0.05);
			\end{tikzpicture}
			&
			\begin{tikzpicture}[scale = 1.2]
			\draw[thick] (0, 0) -- (0.5,0);
			\draw[thick][white] (0.5, 0) -- (0.5, 0.5);
			\draw[thick, white] (0.5, 0) -- (1, 0);
			\draw[thick] (0.5, -0.5) -- (0.5, 0);
			(0.5,0.5);
			\draw[fill] (0.5, 0) circle (0.05);
			\end{tikzpicture}
						&
			\begin{tikzpicture}[scale = 1.2]
			\draw[thick, white] (0, 0) -- (1,0);
			\draw[thick][white] (0.5, -0.5) -- (0.5,0.5);
			\draw[fill] (0.5, 0) circle (0.05);
			\end{tikzpicture}
			&
			\begin{tikzpicture}[scale = 1.2]
			\draw[thick] (0.5, 0) -- (1,0);
			\draw[thick] (0.5, 0) -- (0.5, 0.5);
			\draw[thick][white] (0, 0) -- (0.5, 0);
			\draw[thick][white] (0.5, -0.5) -- (0.5, 0);
			(0.5,0.5);
			\draw[fill] (0.5, 0) circle (0.05);
			\end{tikzpicture}
			\\
			\hline
			Weight 
			& 1 & 1 & $b_1$ & $1- b_1$ & $b_2$ & $1 - b_2$\\
			\hline
\end{tabular}
\caption{Six types of configurations for the S6V model after horizontal complementation.}
	\label{fig:cs6v}
\end{figure}

As done in \cite{aggarwal2021deformed}, if we scale the weights  $b_1 \to t$ and $b_2 \to 1$ in an appropriate way and simultaneously scale the discrete lattice to the continuum with certain boundary data, the complemented model in the previous paragraph converges to the $t$-PNG model. One can observe that the weights in Figure \ref{fig:cs6v} reduce to that of the $\LL^1$ matrix in the scaling limit.

It is natural to ask if our colored $t$-PNG model is related to the colored S6V model \cite{borodin2018coloured}. We focus on the case with two colors and recall the definition of the two-colored S6V model from \cite[Section 2]{borodin2018coloured} (the multicolored S6V model was also defined therein). As usual, we use red to denote the higher priority color and blue to denote the lower priority color. The number of output lines for each color must equal the number of input lines for that color; however, unlike for the colored $t$-PNG model, there can be at most one line emanating from the vertex in each direction. The vertex weight of a two-colored S6V configuration is then defined to be the weight of the single-colored S6V configuration obtained by just considering the lines of the highest-priority color present in the configuration.
\begin{figure}[ht]
    \centering
\begin{tikzpicture}
\begin{scope}[xshift = 0]
\foreach \x in {1, 2, 3, 4, 5}
{
\foreach \y in {1, 2, 3, 4, 5}
\draw[fill] (\x, \y) circle (0.05);
}

\draw (0.5, 1) -- (1, 1) -- (1, 3) -- (2, 3) -- (2, 4) -- (4, 4) -- (4, 5.5);
\draw (2, 0.5) -- (2, 2) -- (5.5, 2);
\draw (4, 0.5) -- (4, 3) -- (5.5, 3);
\draw (0.5, 5) -- (3, 5) -- (3, 5.5);
\node at (0.6, 0.6) {$(0, 0)$};
\end{scope}
\begin{scope}[xshift = 7cm]
\foreach \x in {1, 2, 3, 4, 5}
{
\foreach \y in {1, 2, 3, 4, 5}
\draw[fill] (\x, \y) circle (0.05);
}

\draw (1, 1) -- (5.5, 1);
\draw (0.5, 2) -- (2, 2);
\draw (0.5, 3) -- (1, 3);
\draw (2, 3) -- (4, 3);
\draw (0.5, 4) -- (2, 4);
\draw (4, 4) -- (5.5, 4); 
\draw (3, 5) -- (5.5, 5);
\draw (1, 1) -- (1, 3);
\draw (2, 0.5) -- (2, 2);
\draw (2, 3) -- (2, 4);
\draw (4, 0.5) -- (4, 3);
\draw (3, 5) -- (3, 5.5);
\draw (4, 4) -- (4, 5.5);
\node at (0.6, 0.6) {$(0, 0)$};
\end{scope}
\end{tikzpicture}
\caption{Left panel: A sampling of the S6V model on the first quadrant. Right panel: The S6V model after complementing the horizontal lines.}
\label{fig:sampling}
\end{figure}
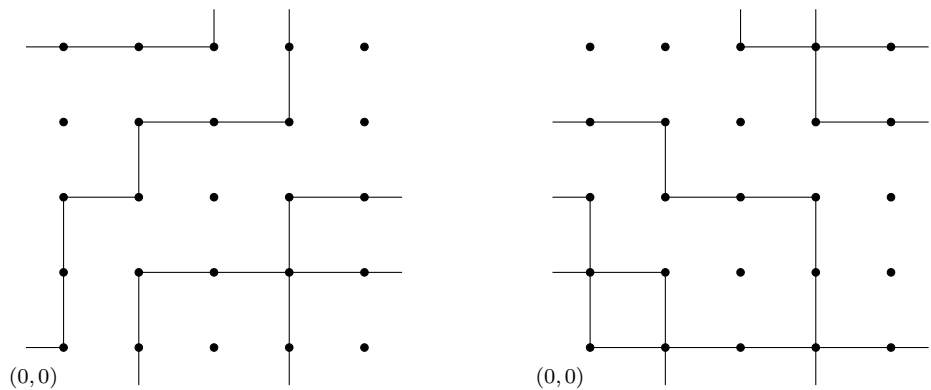

There is also a relation between the two-colored $t$-PNG model and the two-colored S6V model via horizontal complementation. Given a two-colored $t$-PNG configuration with at most one line in the vertical direction, we can perform the following horizontal complementation (referred to as ``hc" in the figures below). If there is a horizontal \textbf{red} line, we erase it; if there is no horizontal red line, we add it. We leave the blue lines unchanged. One can check that the resulting configuration is a configuration for the two-colored S6V model where $b_1 = t$ and $b_2 = 1$ and that the two configurations have the same weights in their respective models. However, this procedure will not work for two-colored $t$-PNG configurations with both blue and red lines in the vertical direction (see Figure \ref{fig:two-colored}). This is because the complemented configuration would still have two lines in the vertical direction, while the two-colored S6V model allows at most one line in each direction. Therefore, not every two-colored $t$-PNG configuration can be obtained through horizontal complementation of the two-colored S6V model.

\begin{figure}[ht]
\centering
\begin{tikzpicture}
\begin{scope}[xshift= 3cm]
\draw[thick, blue] (-0.5, 0) -- (0.5, 0); 
\draw[thick, red] (0, -0.5) -- (0, 0.5); 
\node at (0, -1) {$t$};
\end{scope}
\node at (1.5, 0) {$\overset{\text{hc}}{\longleftrightarrow}$};
\begin{scope}[xshift = 0cm]
\draw[thick, blue] (-0.5, -0.05) -- (0.5, -0.05);
\draw[thick, red] (-0.5, 0.05) -- (0.5, 0.05);
\draw[thick, red] (0, -0.5) -- (0, 0.5); 
\node at (0, -1) {$t$};
\end{scope}
\node at (0, 1) {$t$-PNG};
\node at (3, 1) {S6V};
\begin{scope}[xshift = 9cm]
\draw[thick, red] (-0.05, -0.5) -- (-0.05, 0.5);
\draw[thick, blue] (0.05, -0.5) -- (0.05, 0.5); 
\node at (0, -1) {$1$};
\node at (0, 1) {$t$-PNG};
\node at (1.5, 0) {$\overset{\text{hc}}{\longleftrightarrow}$};
\end{scope}

\begin{scope}[xshift = 12cm]
\draw[thick, red] (-0.05, -0.5) -- (-0.05, 0.5);
\draw[thick, blue] (0.05, -0.5) -- (0.05, 0.5); 
\draw[thick, red] (-0.5, 0) -- (0.5, 0);
\node at (0, -1) {invalid};
\node at (0, 1) {S6V};
\end{scope}
\end{tikzpicture}
\caption{Left panel: When there is at most one line in the vertical direction, one can obtain the two-colored $t$-PNG model by horizontally complementing the red lines. Moreover, the weights of the two-colored  $t$-PNG and S6V configurations match. Right panel: When there are two lines in the vertical direction, one cannot obtain the $t$-PNG model via horizontal complementation, since the S6V model allows at most one line in the vertical direction.}
\label{fig:two-colored}
\end{figure}
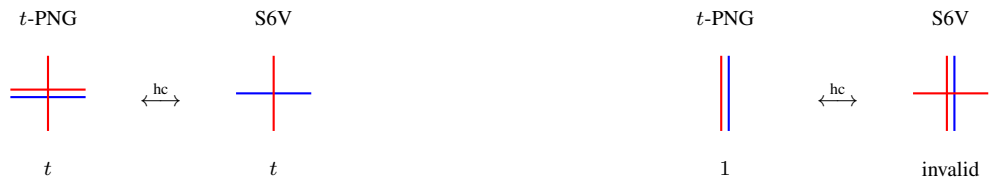

\subsection{A Burke's theorem for the $t$-PNG model}
The original Burke’s theorem \cite{burke1956output} states that the departure process
of an M/M/1 queue with a Poisson arrival process is a Poisson process. Generalizations of Burke's theorem have been proved for last passage models \cite{HammersleySourcesAndSinks, balazs2006cube, seppalainen2017variational}, polymers \cite{o2001brownian, seppalainen2012scaling, chaumont2018characterizing}, and stochastic vertex models \cite{aggarwal2018current, imamura2020stationary, lin2020kpz}. Note that these models all admit a stationary version and that Burke's theorem is a stronger property than stationarity.

The $t$-PNG model has a stationary version, which will be an important tool in the proof of our main theorem and is constructed in Section \ref{sec:upper-bound}. As a natural extension,  we also prove a Burke's theorem for the $t$-PNG model. Our result extends Theorem 3.1 in \cite{HammersleySourcesAndSinks} from $t = 0$ to general $t \in [0, 1)$.

\begin{theorem}[Burke's Theorem] \label{thm:burke}
Fix $T_1, T_2 > 0$. Consider the $t$-PNG model on $[0, T_1] \times [0, T_2]$ with a Poisson process of sources on of intensity $\lambda$ on the bottom boundary, a Poisson process of sinks of intensity $\frac{1}{\lambda(1-t)}$ on the left boundary, and Poisson process of nucleations of intensity $1$ in the interior of the box. We choose all three Poisson processes to be independent. Denote this process by $L_{\lambda}$ and let $L_{\lambda}^{\text{corner}}$ denote the set of corner points, let $L_{\lambda}^{\text{in}}$ denote the entry points of paths on the right boundary, and let $L_{\lambda}^{\text{out}}$ denote the exit points of paths on the top boundary. Then 
\begin{enumerate}[leftmargin = 2em, label = (\roman*)]
\item \label{item:stationary}
$L_{\lambda}^{\text{corner}}$ is a Poisson point process with intensity $1$ in $[0, T_1] \times [0, T_2]$, $L_{\lambda}^{\text{in}}$ is a Poisson point process of intensity $\frac{1}{\lambda(1-t)}$, and $L_{\lambda}^{\text{out}}$ is a Poisson point process of intensity $\lambda$. 
\item All three Poisson processes are independent.
\end{enumerate}
\end{theorem}
The proof of Theorem \ref{thm:burke} is given in Section \ref{sec:upper-bound}. 

\subsection{Proof idea}
As explained after Remark \ref{rmk:localconvergence}, unlike in the case of the PNG model, we can no longer directly apply the superadditive ergodic theorem to the $t$-PNG model using attractivity. Instead, we need to define a colored version of the $t$-PNG model which will allow us to construct a family of random variables $\{X_{m, n}: 0 \leq m \leq n\}$ that satisfies the conditions of the superadditive ergodic theorem.

Definition \ref{def:nLmatrix} for the matrices $\{\LL^n: n \geq 1\}$ follows from a sophisticated guess based on Properties \hyperlink{hyp:property1}{1}--\hyperlink{hyp:property3}{3};
however, it is not immediately clear that these matrices are stochastic, and it also takes some work to show that all of the desired Properties \hyperlink{hyp:property1}{1}--\hyperlink{hyp:property3}{3} are indeed satisfied. The proof of stochasticity follows from an induction argument. Properties \hyperlink{hyp:property1}{1}--\hyperlink{hyp:property3}{3} then follow from a case-by-case study using the stochasticity and the fact that the entries of $\{\LL^n: n \geq 1\}$ belong to $\{0, t,1-t, 1\}$. 

Having defined the stochastic matrices $\{\LL^n: n \geq 1\}$, we proceed to prove the hydrodynamic limit theorem.  
The idea is to consider a colored $t$-PNG model and define $X_{m, n}$ to be a color-dependent version of the height function restricted to $[m, n] \times [m, n]$. We assign the Poisson nucleations with different colors so that points with smaller $x$-coordinates or $y$-coordinates have lower priority. We can now define $X_{m,n}$ in such a way so that it only depends on nucleations inside the box $[m, n] \times [m, n]$ and ignores all lines entering from the bottom or the left since those lines will have lower priority. Therefore, the random variables $X_{m,n}$ will be stationary and independent on disjoint boxes and hence ergodic. 
The superadditivity intrinsically follows from a coupling argument that says that allowing more lines to enter the boundary only increases the height function. We can now apply the superadditive ergodic theorem \cite{liggett1985improved} to the family of random variables $\{X_{m, n}: 0 \leq m \leq n\}$. By the property of mod 2 erasure, the colored $t$-PNG model reduces to the single-colored $t$-PNG model after color projection. Hence, we have
\begin{equation}\label{eq:lln}
\lim_{n \to \infty} \frac{N(n,n)}{n} = \sup_{n \geq 1} \frac{\mathbb{E}[N(n,n)]}{n}=\gamma, \qquad \text{ a.s.}, 
\end{equation}
where $\gamma$ is some constant in $[0, \infty].$ 

To determine the value of this constant, we adopt an idea from \cite{groeneboom2001ulam}. 
Before carrying out that idea in the next paragraph, we need to first show that $\gamma \in (0, \infty)$. We do this by deriving an upper bound and lower bound for $\gamma$. The lower bound can be simply obtained via a coupling with the PNG model. To obtain the upper bound, we generalize \cite{aldous-diaconis} and view the $t$-PNG model as an interacting particle system where the vertical lines play the role of particle trajectories. 
We add sources and sinks on the left and bottom boundary and obtain a stationary version of the $t$-PNG model. 
A coupling with the stationary model together with the attractive property provides an upper bound for $\gamma$.

In \cite{groeneboom2001ulam}, the author proposed a soft argument for computing the constant $\gamma$ in the case of $t = 0$. The key idea is to relate the computation of $\gamma$ to  the 
law of large numbers (LLN) for the number of $\alpha$-points. When $t = 0$, the $\alpha$-points are exactly the Poisson nucleations and the LLN follows immediately.  For $t > 0$, 
the set of $\alpha$-points is the union of 
Poisson nucleations and crossing points. We prove an LLN 
for the number of $\alpha$-points using the fact that the number of $\alpha$-points on the horizontal line emanating from a Poisson nucleation is asymptotically a geometric random variable.

\subsection{Literature review}
In \cite{MR0129165}, Ulam first posed the question of studying the average length of the longest increasing subsequence of a random permutation, which is now called ``Ulam's problem". 
In \cite{Seedlings}, Hammersley transferred the problem into a last-passage model now called the polynuclear growth (PNG) model and used the subadditive ergodic theorem to show that the limit $\gamma$ exists. 
Logan and Shepp \cite{Logan-Shepp}, and Vershik and Kerov \cite{Vershik-Kerov} simultaneously proved that $\gamma = 2$. In fact, they proved a more general result concerning the limit shape of a Young diagram associated with the random permutation. 
In \cite{aldous-diaconis}, an alternative proof is given using the perspective of interacting particle systems. \cite{baik1999distribution} proved the entire fluctuation theorem through an analysis of exact formulas. \cite{seppalainen1996microscopic, groeneboom2001ulam, groeneboom2002hydrodynamical, HammersleySourcesAndSinks, basdevant2015discrete} gave different proofs of $\gamma = 2$ using soft arguments. 

Beyond the PNG model, the limit shapes of other last passage models have been well studied. For two discrete variants of the PNG model, the almost sure convergence to an explicit limit shape has been proved in \cite{seppalainen1997increasing, seppalainen1998exact, basdevant2015discrete}. The explicit limit shape for the exponential last passage percolation (LPP) model was first proved in \cite{rost1981non} (also see the notes \cite{seppalainen2009lecture, seppalainen2017variational, ferrari2018tasep}), and similar results can be obtained for the geometric case. The fluctuations of the exponential and geometric LPP models around their limit shapes are given in \cite{johansson2000shape} using integrable methods. We remark that a common approach to deriving the explicit limit shape of an exactly solvable LPP  or polymer model is to couple it with a stationary model and then solve a corresponding variational problem. This idea, however, no longer works for the $t$-PNG model since we do not have a similar coupling. For the LPP model with general i.i.d. weights, the limit shape is no longer explicit. \cite{martin2004limiting} proved a general shape theorem and studied the continuity of the shape function as well as its asymptotic behavior near the edge. A shape theorem for the last passage percolation model on a
two-dimensional compound Poisson process was proved in \cite{cator2010shape}.

For the PNG model, the two-colored version has been studied in \cite{HammersleySourcesAndSinks}, where the paths of the second color can be viewed as the trajectories of the second class particles. The behavior of second class particles in the PNG model was studied in \cite{cator2006behavior}. By a duality between the second class particle and the exit point, \cite{cator2006second} shows that the fluctuation exponent of the stationary PNG model along its characteristic direction is $\frac{1}{3}$. A similar result was obtained in \cite{ciech2019order} for a discrete variant of the PNG model. The colored PNG model has been considered in \cite{ferrari2009multiclass}, where it was obtained via the basic coupling of multiple particles. Note that it is not clear how to go from this perspective to our definition of the stochastic matrices in Definition \ref{def:Lmatrix}.

The authors of \cite{aggarwal2021deformed} introduced the $t$-PNG model, which is a one-parameter deformation of the PNG model.
They proved a one-point fluctuation limit theorem for the $t$-PNG model using integrable methods. They also proved one-point convergence to the KPZ equation. We remark that a different deformation of the $t$-PNG model was considered in \cite{pei2016q}. The $t$-PNG model can be realized as a scaling limit of the S6V and its higher spin generalization after complementing the horizontal lines. The S6V model and its various generalizations have been studied recently in \cite{corwin2016stochastic, kuniba2016stochastic, corwin2017kpz, borodin2018higher, borodin2018coloured, kuan2018algebraic,    borodin2020observables, lin2020kpz,  lin2020stochastic, imamura2020stationary, aggarwal2021colored, aggarwal2021deformed} and references therein. 
\subsection*{Outline}
In Section \ref{sec:multicolor} we prove that the matrices in Definition \ref{def:Lmatrix} are stochastic and satisfy Properties \hyperlink{hyp:property1}{1}-\hyperlink{hyp:property2}{2}. In Section \ref{sec:subadditivity} we apply the superadditive ergodic theorem to prove the hydrodynamic limit without identifying the constant $\gamma$. In Section \ref{sec:upper-bound} we prove an upper bound for $\gamma$ by constructing a stationary version of the $t$-PNG model, we show that Property \hyperlink{hyp:property3}{3} is satisfied, and we also prove a version of Burke's Theorem. In Section \ref{sec:proving} we identify $\gamma$ by proving a law of large numbers for the number of $\alpha$-points. In Appendix \ref{sec:twocolorfig}, we give all the two-colored configurations with non-zero weights. In Appendix \ref{sec:generators}, we provide some technical computations for Section \ref{sec:upper-bound}. Appendix \ref{sec:simulations} provides computer simulations of the $t$-PNG model for different values of $t$.


\section{Properties of the $\LL$ matrices} \label{sec:multicolor}
In this section, we prove that the matrices $\{\LL^n: n \geq 1\}$ defined in Definition \ref{def:nLmatrix} are stochastic and satisfy the properties of color ignorance and mod 2 erasure. 
The stochasticity is shown in Proposition \ref{prop:stochasticity}. The properties of color ignorance and mod 2 erasure are respectively shown in Propositions \ref{prop:colorignore} and \ref{prop:mod2erasure}. 
\begin{lemma}\label{lem:uniqueoutput1}
Fix arbitrary $i, j \in \{0, 1\}$. There exists a unique pair $(k, l) \in \{0, 1\}^2$ such that $\LL^1 (i, j; k, l) \in \{t, 1\}$.  Similarly, there exists a unique pair $(k', l') \in \{0, 1\}^2$ such that $\LL^1 (i, j; k', l') \in \{1-t, 1\}$. 
\end{lemma}
\begin{proof}
This is straightforward from Definition \ref{def:Lmatrix}.
\end{proof}
Let $A = \{r_1, \dots, r_m\}$ be an ordered subset of $\{1, \dots, n\}$. For $\bx = (x_1, \dots, x_n) \in \{0, 1\}^n$, we define $\bx_A = (x_{r_1}, \dots, x_{r_m})$.  In particular, we define $\bx_{[1, m]} = (x_1, \dots, x_m)$.
\begin{proposition}[Stochasticity]\label{prop:stochasticity}
For arbitrary fixed $n \in \Z_{\geq 1}$, 
$\LL^n$ is a stochastic matrix, i.e. the entries of $\LL^n$ are non-negative, and for any $\bi, \bj \in \{0, 1\}^n$, 
\begin{equation*}
\sum_{\bk, \bl \in \{0, 1\}^n} \LL^n(\bi, \bj; \bk, \bl) = 1.
\end{equation*}
\end{proposition}
\begin{proof}
Let us prove the result by induction. When $n = 1$, 
we can easily see that $\LL^1$ is stochastic. We use induction to prove the stochasticity for general $n$. 
We assume  $\LL^{N-1}$ is stochastic 
and use this to show that $\LL^N$ is also stochastic. Referring to \eqref{eq:weight}, for $\bi, \bj, \bk, \bl \in \{0, 1\}^N$, we have
\begin{equation}\label{eq:Lrecur}
\LL^N(\bi, \bj; \bk, \bl) =
\mmin\Big(\LL^{N-1}\big(\bi_{[1, N-1]}, \bj_{[1, N-1]}; \bk_{[1, N-1]}, \bl_{[1, N-1]}\big), \LL^1(\mfs_N (\bi), \mfs_N (\bj); \mfs_N (\bk), \mfs_N (\bl))\Big). 
\end{equation}
Fix arbitrary $\bi, \bj \in \{0, 1\}^N$. $\LL^{N-1}$ is a stochastic matrix whose entries take value in $\{0, t, 1-t, 1\}$, so we have one of the following two cases:
\bigskip
\\
\textbf{Case 1:} There exist unique $\mfk, \mfl \in \{0, 1\}^{N-1}$ such that $\LL^{N-1}(\bi_{[1, N-1]}, \bj_{[1, N-1]}; \mfk, \mfl) = 1$.\\ 
\textbf{Case 2:} There exist unique $\mfk_1, \mfl_1, \mfk_2, \mfl_2 \in \{0, 1\}^{N-1}$ such that 
\begin{equation*}
\LL^{N-1}(\bi_{[1, N-1]}, \bj_{[1, N-1]}; \mfk_1, \mfl_1) = t, \qquad \LL^{N-1}(\bi_{[1, N-1]}, \bj_{[1, N-1]}; \mfk_2, \mfl_2) = 1 - t.
\end{equation*}

Let us prove the stochasticity of $\LL^N$ in each case. 
\bigskip
\\
\textbf{Proof for Case 1:}
Using $\LL^{N-1}(\bi_{[1, N-1]}, \bj_{[1, N-1]}; \mfk, \mfl) = 1$ and \eqref{eq:Lrecur}, we have $$\LL^N (\bi, \bj; \bk, \bl) = \LL^1(\mfs_N(\bi), \mfs_N(\bj); \mfs_N (\bk), \mfs_N (\bl))$$ when $\bk_{[1, N-1]} = \mfk$ and $\bl_{[1, N-1]} = \mfl$. In addition, $\LL^{N}(\bi, \bj; \bk, \bl) = 0$ when $\bk_{[1, N-1]} \neq \mfk$ or $\bl_{[1, N-1]} \neq \mfl$. Note that $(\mfs_N(\bk), \mfs_N (\bl))$ equals each element of $\{0, 1\}^2$ exactly once when we vary $\bk, \bl$ under the restriction $\bk_{[1, N-1]} = \mfk$ and $\bl_{[1, N-1]} = \mfl$. Therefore,
\begin{equation*}
\sum_{\bk, \bl \in \{0, 1\}^N} \LL^N(\bi, \bj; \bk, \bl) = \sum_{\substack{\bk_{[1, N-1]} = \mfk\\ \bl_{[1, N-1]} = \mfl}} \LL^1 (\mfs_N(\bi), \mfs_N(\bj); \mfs_N (\bk), \mfs_N (\bl)) = 1.
\end{equation*}
The last equality comes from the stochasticity of $\LL^1$.
\bigskip
\\
\textbf{Proof for Case 2:}  
By \eqref{eq:Lrecur}, $\LL^N(\bi, \bj; \bk, \bl) = 0$ when $(\bk_{[1, N-1]}, \bl_{[1, N-1]}) \notin \{(\mfk_1, \mfl_1), (\mfk_2, \mfl_2)\}.$ If $(\bk_{[1, N-1]}, \bl_{[1, N-1]}) = (\mfk_1, \mfl_1)$, we have 
$$\LL^N(\bi, \bj; \bk, \bl) = \mmin \big(t, \LL^1 (\mfs_N (\bi), \mfs_N (\bj); \mfs_N (\bk), \mfs_N (\bl))\big).$$ 
By Lemma \ref{lem:uniqueoutput1}, there exists only one pair of $(\mfs_N(\bk), \mfs_N(\bl))$ such that $\LL^1(\mfs_N (\bi), \mfs_N (\bj); \mfs_N (\bk), \mfs_N (\bl)) \in \{t, 1\}$.
Using this and the equation above, 
there exists a unique pair $(\bk^   1, \bl^1) \in \{0, 1\}^N$ such  that $\bk^1_{[1, N-1]} = \mfk_1$ and $\bl^1_{[1, N-1]} = \mfl_1$ and $\LL^{N}(\bi, \bj; \bk, \bl) = t$.
Similarly, there exists only one pair of $(\bk^2, \bl^2) \in \{0, 1\}^N$ such that $\bk^2_{[1, N-1]} = \mfk_2$, $\bl^2_{[1, N-1]} = \mfl_2$ and $\LL^N(\bi, \bj; \bk^2, \bl^2) = 1 - t$.  Hence, we have
\begin{equation*}
\sum_{\bk, \bl \in \{0, 1\}^N} \LL^N(\bi, \bj; \bk, \bl) = \LL^N(\bi, \bj; \bk^1, \bl^1) + \LL^N(\bi, \bj; \bk^2, \bl^2) = 1. 
\end{equation*} 
By \eqref{eq:weight}, $\LL^N$ is non-negative, hence it is a stochastic matrix.
\end{proof}

The following  lemmas will be used to prove the properties of color ignorance and mod 2 erasure. 
\begin{lemma}\label{lem:uniqueoutput2}
	Fix positive integers $m \leq n$. 
	Assume that  for fixed $\mfi, \mfj, \mfk, \mfl \in \{0, 1\}^m$ we have $\LL^m(\mfi, \mfj; \mfk, \mfl) \in \{t, 1-t\}$. 
Then for any fixed $\bi, \bj \in \{0, 1\}^n$ satisfying $(\bi_{[1, m]}, \bj_{[1, m]}) = (\mfi, \mfj)$, there exist unique $\bk^1, \bl^1$, $\bk^2, \bl^2 \in \{0, 1\}^n$ satisfying 
	\begin{equation*}
	\LL^n(\bi, \bj; \bk^1, \bl^1) = \q, \qquad \LL^n(\bi, \bj; \bk^2, \bl^2) = 1 - \q.
	\end{equation*}
	For $(\bk, \bl)$ that does not equal either $(\bk^1, \bl^1)$ or $(\bk^2, \bl^2)$, we have $\LL^n(\bi, \bj; \bk, \bl) = 0$. 
\end{lemma}
\begin{proof}
When $n = m$, the claim is true due to the stochasticity of   $\mathsf{L}^m$. Now we prove the lemma for $m < n$.
Using the stochasticity of $\LL^m$, we know that there  exist $\mfk^1, \mfl^1; \mfk^2, \mfl^2 \in \{0, 1\}^m$ satisfying $\LL^m(\mfi, \mfj; \mfk^1, \mfl^1) = t$ and $\LL^m(\mfi, \mfj; \mfk^2, \mfl^2) = 1 - t$. 
Using \eqref{eq:weight} and  $\bi_{[1, m]} = \mfi,  \bj_{[1, m]} = \mfj$, we have
\begin{equation}\label{eq:minn}
\LL^{n}(\bi, \bj; \bk, \bl) = \mmin\bigg(\LL^m(\mfi, \mfj; \bk_{[1, m]}, \bl_{[1, m]}), \mmin_{r = m+1}^n\Big(\LL^1 1(\mathfrak{s}_{r} (\bi), \mathfrak{s}_{r} (\bj); \mathfrak{s}_{r} (\bk), \mathfrak{s}_{r} (\bl))\Big)\bigg).
\end{equation}
If $(\bk_{[1, m]}, \bl_{[1, m]}) = (\mfk^1, \mfl^1)$, we have 
	\begin{equation*}
	\LL^n(\bi, \bj; \bk, \bl) = \mmin\bigg(t,  \mmin_{r = m+1}^n\Big(\LL^1(\mathfrak{s}_{r} (\bi), \mathfrak{s}_{r} (\bj); \mathfrak{s}_{r} (\bk), \mathfrak{s}_{r} (\bl))\Big)\bigg).
	\end{equation*}
	Therefore, $\LL^n(\bi, \bj; \bk, \bl) = t$ if and only if $\LL^1(\mathfrak{s}_{r} (\bi), \mathfrak{s}_{r} (\bj); \mathfrak{s}_{r} (\bk), \mathfrak{s}_{r} (\bl)) \in \{t, 1\}$ for every $r \in \{m+1, \dots, n\}$.
	By Lemma \ref{lem:uniqueoutput1}, for every $\mathfrak{s}_{r} (\bi), \mathfrak{s}_{r} (\bj) \in \{0, 1\}$, there exist unique $\mathfrak{s}_{r} (\bk), \mathfrak{s}_{r} (\bl)\in \{0, 1\}$ such that $\LL^1(\mathfrak{s}_{r} (\bi), \mathfrak{s}_{r} (\bj); \mathfrak{s}_{r} (\bk), \mathfrak{s}_{r} (\bl)) \in \{t, 1\}$. Since 
	$(\bk, \bl)$ is uniquely determined by the value of $(\bk_{[1, m]}, \bl_{[1, m]})$ and $(\mfs_r(\bk), \mfs_{r} (\mathbf{\bl}))_{r = m+1}^n$ (and vice versa), there is a unique pair $(\bk^1, \bl^1)$ such that $(\bk^1_{[1, m]}, \bl^1_{[1, m]}) = (\mfk^1, \mfl^1)$ and 
	$\LL^n(\bi, \bj; \bk^1, \bl^1) = t$. For  $(\bk, \bl)$ satisfying $(\bk_{[1, m]}, \bl_{[1, m]}) = (\mfk^1, \mfl^1)$ and $(\bk, \bl) \neq (\bk^1, \bl^1)$, we have $\LL^n(\bi, \bj; \bk, \bl) = 0$.
	
For a similar reason,	there exists a unique pair $(\bk^2, \bl^2)$ such that $(\bk^2_{[1, m]}, \bl^2_{[1, m]}) = (\mfk^2, \mfl^2)$ and $\LL^n(\bi, \bj; \bk^2, \bl^2) = 1-t$.  By the stochasticity of $\LL^n$, for $(\bk, \bl) \notin \{(\bk^1, \bl^1), (\bk^2, \bl^2)\}$, we have $\LL^n(\bi, \bj; \bk, \bl) = 0$. 
This concludes the lemma.
\end{proof}

Let us prepare some notation for the next lemma. We call  $\pi$ a partition of $\{1, \dots, n\}$ if  
$\pi$ takes the form of  
$$\pi = \big\{\{1, \dots, r_1\}, \{r_1 + 1, \dots, r_2\}, \dots, \{r_{m-1} + 1, \dots, r_m\}\big\}$$
for some $m \leq n$ and  $1 = r_1 < r_2 < \dots < r_m  = n$. We define a map $g_\pi$ from $\{0, 1\}^n \to \{0, 1\}^m$ such that 
\begin{equation*}
g_{\pi} (x_1, \dots, x_n) = \bigg(\Big(\sum_{i = r_{k-1}+1}^{r_k} x_i\Big) \text{ mod } 2\bigg)_{k = 1}^m.
\end{equation*}
We define $\ell(\pi) = m$ to be the length of the partition $\pi$.
\begin{lemma}\label{lem:uniqueoutput3} 
Fix positive integers $m \leq n$. Fix a partition  $\pi$ of $\{1, \dots, n\}$ such that $\ell(\pi) = m$. 
Assume that  for fixed $\mfi, \mfj, \mfk, \mfl \in \{0, 1\}^m$, we have $\LL^m(\mfi, \mfj; \mfk, \mfl) \in \{t, 1-t\}$. 
Then for any fixed $\bi, \bj \in \{0, 1\}^n$ satisfying $(g_\pi(\bi), g_\pi(\bj)) = (\mfi, \mfj)$, there exist $\bk^1, \bl^1$, $\bk^2, \bl^2 \in \{0, 1\}^n$ satisfying 
\begin{equation*}
\LL^n(\bi, \bj; \bk^1, \bl^1) = \q, \qquad \LL^n(\bi, \bj; \bk^2, \bl^2) = 1 - \q.
\end{equation*}
For $(\bk, \bl)$ that does not equal either $(\bk^1, \bl^1)$ or $(\bk^2, \bl^2)$, we have $\LL^n(\bi, \bj; \bk, \bl) = 0$. 
\end{lemma}
\begin{proof}
The lemma is clearly true when $m = n$. We only need to prove it for $m < n$.

By stochasticity of $\LL^m$, there exists  $\mfk^1, \mfl^1, \mfk^2, \mfl^2 \in \{0, 1\}^m$ satisfying $\LL^m(\mfi, \mfj, \mfk^1, \mfl^1) = t$ and $\LL^m(\mfi, \mfj, \mfk^2, \mfl^2) = 1 - t$. 
We let 
	\begin{equation*}
	\pi = \big\{\{1, \dots, r_1\}, \{r_1 + 1, \dots, r_2\}, \dots, \{r_{m-1} + 1, \dots, r_m\}\big\}
	\end{equation*}
where $r_m = n$. Let $A_\pi = \{r_1, \dots, r_m\}$ and $B_\pi = \{1, \dots, n\} \backslash A_\pi$.
	Using \eqref{eq:weight}, we have
	\begin{align*}
	\LL^{n}(\bi, \bj; \bk, \bl) &= \mmin_{r=1}^n \big(\LL^1(\mfs_r(\bi), \mfs_r (\bj); \mfs_r (\bk), \mfs_r (\bl))\big)\\ 
	&=  \mmin\Big(\mmin_{r \in A_\pi}\big(\LL^1(\mfs_r(\bi), \mfs_r (\bj); \mfs_r (\bk), \mfs_r (\bl))\big), \mmin_{r \in B_\pi} \big(\LL^1(\mfs_r(\bi), \mfs_r (\bj); \mfs_r (\bk), \mfs_r (\bl))\big)\Big)\\
		&=\mmin\Big(\LL^m(g_\pi (\bi), g_\pi (\bj); g_\pi (\bk), g_\pi (\bl)),  \mmin_{r \in B_\pi}\big(\LL^1(\mfs_r(\bi), \mfs_r (\bj); \mfs_r (\bk), \mfs_r (\bl))\big) \Big).
	\end{align*}
Note that $(g_\pi (\bi), g_\pi (\bj)) = (\mfi, \mfj)$. If $(g_\pi (\bk), g_\pi (\bl)) = (\mfk^1, \mfl^1)$, then
\begin{equation*}
\LL^n(\bi, \bj; \bk, \bl) = \mmin\Big(t, \mmin_{r \in B_\pi}\big(\LL^1(\mfs_r(\bi), \mfs_r (\bj); \mfs_r (\bk), \mfs_r (\bl))\big) \Big).
\end{equation*}

In order for $\LL^n(\bi, \bj; \bk, \bl) = t$, we need $\LL^1(\mfs_r (\bi), \mfs_r (\bj); \mfs_r (\bk), \mfs_r (\bl)) \in \{t, 1\}$ for each $r \in B_\pi$; otherwise, $\LL^n(\bi, \bj; \bk, \bl) = 0$. Note that $\bi, \bj$ are fixed. By Lemma \ref{lem:uniqueoutput1}, for each $r \in B_\pi$,  there is only one choice for $(\mfs_r (\bk), \mfs_r (\bl))$ such that $\LL^1(\mfs_r (\bi), \mfs_r (\bj); \mfs_r (\bk), \mfs_r (\bl)) \in \{t, 1\}$. It is straightforward that we have a bijection from $\{0, 1\}^n$ to itself, given by $\bx \leftrightarrow (g_\pi (\bx), (\mfs_r (\bx))_{r \in B_\pi})$. Hence, there exist  unique $\bk^1, \bl^1$ 
satisfying $(g_\pi (\bk^1), g_\pi (\bl^1)) = (\mfk^1, \mfl^1)$ and 
$\LL^n(\bi, \bj; \bk^1, \bl^1) = t$. 

Similarly, there exist unique $\bk^2, \bl^2$ such that $(g_\pi (\bk^2), g_\pi (\bl^2)) = (\mfk^2, \mfl^2)$ and 
$\LL^n(\bi, \bj; \bk^2, \bl^2)=1-t$. By stochasticity, for $(\bk, \bl) \notin \{(\bk^1, \bl^1), (\bk^2, \bl^2)\}$, we have $\LL^n(\bi, \bj; \bk, \bl) = 0$. This concludes the lemma. 
\end{proof}
\begin{proposition}[Color Ignorance]\label{prop:colorignore}
Fix $m \in \{1, \dots, n\}$ and 
$\mfi, \mfj, \mfk, \mfl \in \{0, 1\}^m$. For all $\bi, \bj \in \{0, 1\}^n$ such that $\bi_{[1, m]} = \mfi$ and $\bj_{[1, m]} = \mfj$, we have 
\begin{equation}\label{eq:property1}
\sum_{\substack{\bk_{[1, m]} = \mfk, \\ \bl_{[1, m]} = \mfl}} \LL^n(\bi, \bj; \bk, \bl) = \LL^m(\mfi, \mfj; \mfk, \mfl).
\end{equation}
\end{proposition}

\begin{proof}
	We have $\LL^m(\mfi, \mfj; \mfk, \mfl) \in \{0, t, 1-t, 1\}$. We prove the equality \eqref{eq:property1} for each possible value of $\LL^m(\mfi, \mfj; \mfk, \mfl)$.
\bigskip
\\
\textbf{Case 1: } $\LL^m(\mfi, \mfj; \mfk, \mfl) = 0$. By \eqref{eq:minn}, when $\bi_{[1, m]} = \mfi$, $\bj_{[1, m]} = \mfj$, $\bk_{[1, m]} = \mfk$, $\bl_{[1, m]} = \mfl$, we have $\LL^n(\bi, \bj; \bk, \bl) = 0$. Hence, \eqref{eq:property1} holds. 
\bigskip
\\
\textbf{Case 2: }	$\LL^m(\mfi, \mfj; \mfk, \mfl) = 1$. By stochasticity of $\LL^m$ and \eqref{eq:minn}, if we have $\bi_{[1, m]} = \mfi$, $\bj_{[1, m]} = \mfj$ and $(\bk_{[1, m]}, \bl_{[1, m]}) \neq (\mfk, \mfl)$, we have $\LL^n(\bi, \bj; \bk, \bl) = 0$. Using this and the stochasticity of $\LL^n$, 
	\begin{equation*}
	\sum_{\substack{\bk_{[1, m]} = \mfk\\ \bl_{[1, m]} = \mfl}} \LL^n(\bi, \bj; \bk, \bl) = \sum_{\bk, \bl \in \{0, 1\}^n} \LL^n(\bi, \bj; \bk, \bl) = 1 = \LL^m(\mfi, \mfj; \mfk, \mfl).
	\end{equation*}
	Hence, \eqref{eq:property1} holds. 
\bigskip
\\
\textbf{Case 3: } $\LL^m(\mfi, \mfj; \mfk, \mfl) = \q$. Since $\bi_{[1, m]} = \mfi$ and $\bj_{[1, m]} = \mfj$, by Lemma \ref{lem:uniqueoutput2}, we know that there exist unique $\bk^1, \bl^1 \in \{0, 1\}^n$ satisfying $\bk^1_{[1, m]} = \mfk$, $\bl^1_{[1, m]} = \mfl$ and $\LL^n(\bi, \bj; \bk^1, \bl^1) = \q$. For all $(\bk, \bl) \in \{0, 1\}^n$ satisfying $\bk_{[1, m]} = \mfk$, $\bl_{[1, m]} = \mfk$ and $(\bk, \bl) \neq (\bk^1, \bl^1)$, $\LL^n(\bi, \bj; \bk, \bl) = 0$. Hence,    
	\begin{equation*}
	\sum_{\substack{\bk_{[1, m]} = \mfk\\ \bl_{[1, m]} = \mfl}} \LL^n(\bi, \bj; \bk, \bl) =  \LL^n(\bi, \bj; \bk^1, \bl^1) = t = \LL^m(\mfi, \mfj; \mfk, \mfl).
	\end{equation*}
\textbf{Case 4: $\LL^m(\mfi, \mfj; \mfk, \mfl) = 1-t$}. The proof is similar to Case 3, and we omit it. This concludes the proof of the proposition. 
\end{proof}
\begin{proposition}[Mod 2 Erasure]
	\label{prop:mod2erasure}
Fix a partition $\pi$ of $\{1, \dots, n\}$ such that $\ell(\pi) = m$. Fix $\mfi, \mfj, \mfk, \mfl \in \{0, 1\}^m$. For all $\bi, \bj \in \{0, 1\}^n$ satisfying $g_{\pi} (\bi) = \mfi$ and $g_{\pi} (\bj) = \mfj$, we have \begin{equation}\label{eq:property2}
\sum_{\substack{g_\pi(\bk) = \mfk \\
g_\pi(\bl) = \mfl}} \LL^n(\bi, \bj; \bk, \bl) = \LL^m(\mfi, \mfj; \mfk, \mfl).
 \end{equation}
\end{proposition}
\begin{proof}
Let $\pi = \big\{\{1, \dots, r_1\}, \{r_1 + 1, \dots, r_2\}, \dots, \{r_{m-1} + 1, \dots, r_m\}\big\}$, where $1 \leq r_1 < \dots < r_m = n$. We again divide our proof into four cases.
\bigskip
\\
\textbf{Case 1:  $\LL^m(\mfi, \mfj, \mfk, \mfl) = 0$}. By \eqref{eq:weight}, for $\bi, \bj, \bk, \bl$ satisfying $g_\pi (\bi) = \mfi, g_\pi (\bj) = \mfj, g_\pi (\bk) = \mfk, g_\pi (\bl) = \mfl$, we have 
$$\LL^n(\bi, \bj; \bk, \bl) \leq \mmin_{i = 1}^m \big(\LL^1(\mfs_{r_i} (\bi), \mfs_{r_i} (\bj); \LL^1(\mfs_{r_i} (\bk), \mfs_{r_i} (\bl))\big) = \LL^m(\mfi, \mfj, \mfk, \mfl) = 0.$$ This implies that $\LL^n(\bi, \bj; \bk, \bl) = 0$. Hence, 
\begin{equation*}
\sum_{\substack{g_\pi(\bk) = \mfk \\
g_\pi(\bl) = \mfl}}\LL^n(\bi, \bj; \bk, \bl) = 0 = \LL^m(\mfi, \mfj; \mfk, \mfl).
\end{equation*}
\textbf{Case 2:  $\LL^m(\mfi, \mfj, \mfk, \mfl) = 1$}. By \eqref{eq:weight}, we know that $(g_\pi (\bk), g_\pi(\bl)) \neq (\mfk, \mfl)$ implies that $\LL^n(\bi, \bj; \bk, \bl) = 0$. Therefore, 
\begin{equation*}
\sum_{\substack{g_\pi(\bk) = \mfk, \\
		g_\pi(\bl) = \mfl}} \LL^n(\bi, \bj; \bk, \bl) = \sum_{\bk, \bl \in \{0, 1\}^n} \LL^n(\bi, \bj; \bk, \bl) = 1 = \LL^m(\mfi, \mfj; \mfk, \mfl).
\end{equation*}
The second equality is due to the stochasticity of $\LL^n$. 
\bigskip
\\
\textbf{Case 3:  $\LL^m(\mfi, \mfj, \mfk, \mfl) = t$}. We fix $\bi, \bj$ that satisfy $(g_{\pi} (\bi), g_{\pi} (\bj))  = (\mfi, \mfj)$. By Lemma \ref{lem:uniqueoutput3},  there exist unique $\bk^1, \bl^1$ satisfying  $(g_{\pi} (\bk^1), g_{\pi} (\bl^1)) = (\mfk, \mfl)$ and $\LL^n(\bi, \bj; \bk^1, \bl^1) \neq 0$. Moreover, $\LL^n(\bi, \bj; \bk^1, \bl^1) = t$. Hence, 
\begin{equation*}
\sum_{\substack{g_\pi(\bk) = \mfk, \\
g_\pi(\bl) = \mfl}} \LL^n(\bi, \bj; \bk, \bl) =  \LL^n(\bi, \bj; \bk^1, \bl^1) = t = \LL^m(\mfi, \mfj; \mfk, \mfl).
\end{equation*}
\textbf{Case 4: $\LL^m(\mfi, \mfj, \mfk, \mfl) = 1-t$}. The proof is similar to Case 3.
\end{proof}


\begin{rmk}\label{rmk:colorlabel}
One can also use real numbers to label colors, not just positive integers. The rule is that for two colors with labels $a < b$, the color $b$ has less priority than $a$. Hence, we can still sample the output of an intersection given finite input lines labeled by real numbers, using  $\{\LL^n: n \geq 1\}$. For our application of the colored model in the next section, we use negative integers to label the colors (see also Remark \ref{rmk:negcolor}).
\end{rmk}



\section{The colored model and superadditivity} \label{sec:subadditivity}
We are going to construct $\{X_{m, n}: 0 \leq m \leq n\}$ as discussed in the introduction using the colored $t$-PNG model. Before doing that, let us recall Liggett's superadditive ergodic theorem from \cite{liggett1985improved}. 
Note that the theorem was originally stated in the subadditive setting, but for our purposes, we formulate it in the superadditive setting by placing a negative sign where necessary. 
\begin{theorem}
	\label{thm:liggettergodic}
	Suppose $\{X_{m, n}\}$ is a collection of random variables that is indexed by integers $0 \leq m \leq n$ and satisfies: 
	\begin{enumerate}[leftmargin = 2em, label = (\roman*)]
		\item \label{item:subadditive} Almost surely $X_{0, 0} = 0$ and  $X_{0, n} \geq X_{0, m} + X_{m, n}$ for $0 \leq m \leq n$.
		\item \label{item:ergodic} For each $k \geq 1$, $\{X_{(n-1)k, nk}: n \geq 1\}$ is an ergodic process. 
		\item \label{item:equalind} $\{X_{m, m+k}: k \geq 0\}  \overset{d}{=} \{X_{m+1, m+k+1}: k \geq 0\}$ 
		for each $m \geq 0$.
		\item \label{item:expectation} $\mathbb{E}[X_{0, 1}^-] < \infty$ where $x^{-} = \max(-x, 0)$. 
	\end{enumerate}
	Then there exists 
a constant $\gamma = \sup_{n \geq 1} \frac{\mathbb{E}[X_{0, n}]}{n} \in (-\infty, \infty]$ satisfying   
\begin{align*}
\gamma = \lim_{n \to \infty} \frac{X_{0, n}}{n} \text{ a.s.}
\end{align*}
\end{theorem}

\begin{definition}[Step colored $t$-PNG model]
We consider a Poisson point process of nucleations on $\R_{>0} \times \R_{> 0}$ with density $1$. 
We can assume that there are no nucleations with an integer $x$- or $y$-coordinate and that no two nucleations have the same $x$- or $y$-coordinates, since these events have probability zero. 
Fix arbitrary integers $m, n \in \Z_{\geq 0}$. We color the nucleations inside the unit square $[m, m+1] \times [n, n+1]$ with the color $-\min(m+1, n+1)$. 
In other words, we assign the color $-m-1$ to nucleations lying in the L-shape area $[m, \infty) \times [m, m+1] \cup [m, m+1] \times [m, \infty)$.  
For each nucleation with a given color $k$, the lines emanating from it in the upward and rightward directions also have color $k$. When horizontal and vertical lines intersect, the output lines emanate from the intersection following the stochastic matrix $\{\LL^n: n \geq 1\}$ defined in Definition \ref{def:nLmatrix} (see also Remark \ref{rmk:colorlabel}). 
The model is referred to as the \emph{step colored $t$-PNG model}. See Figure \ref{fig:multi_color}. 
\end{definition}
\begin{rmk}\label{rmk:negcolor}
The reason that we label the colors with negative integers instead of positive integers as in Section \ref{sec:multicolor} is that to apply Theorem \ref{thm:liggettergodic}, we want to construct a model with infinite colors such that nucleations closer to the axes have lower priority. 
\end{rmk}
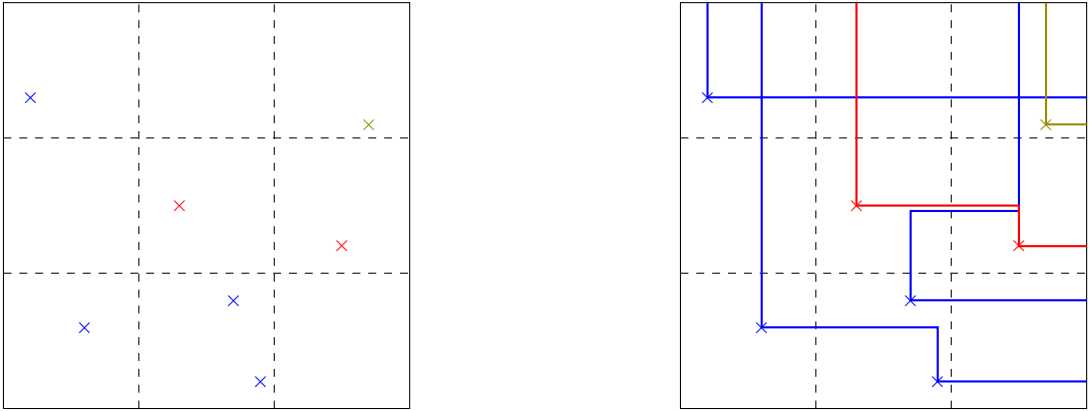
\begin{figure}[b]
\centering

\begin{tikzpicture}[scale = 1.8]
\begin{scope}
\draw(0, 0) -- (3, 0) -- (3, 3) -- (0, 3) -- (0, 0);
\foreach \x in {1, 2}
{\draw[dashed] (0, \x) -- (3, \x);
\draw[dashed] (\x, 0) -- (\x, 3);
}
\node at (0.6, 0.6) {\blue{$\times$}};
\node at (0.2, 2.3) {\blue{$\times$}};
\node at (1.7, 0.8)  {\blue{$\times$}};
\node at (1.9, 0.2) {\blue{$\times$}};
\node at (1.3, 1.5) {\red{$\times$}};
\node at (2.5, 1.2) {\red{$\times$}};
\node at (2.7, 2.1) {\olive{$\times$}};

\draw[thick,blue] (0.2, 3) -- (0.2, 2.3) -- (0.6, 2.3) -- (0.6, 0.6) -- (1.9, 0.6) -- (1.9, 0.2) -- (3, 0.2);
\draw[thick, blue] (0.6, 3) -- (0.6, 2.3) -- (1.7, 2.3)  -- (3, 2.3);
\draw[thick, blue] (3, 0.8) -- (1.7, 0.8) -- (1.7, 1.46) -- (2.5, 1.46) -- (2.5, 3);
\draw[thick, red] (1.3, 3) -- (1.3, 1.5) -- (2.5, 1.5);
\draw[thick, red] (3, 1.2) -- (2.5, 1.2) -- (2.5, 1.5);
\draw[thick, olive] (2.7, 3) -- (2.7, 2.1) -- (3, 2.1);
\end{scope}

\begin{scope}[xshift = -5cm]
\draw(0, 0) -- (3, 0) -- (3, 3) -- (0, 3) -- (0, 0);
\foreach \x in {1, 2}
{\draw[dashed] (0, \x) -- (3, \x);
\draw[dashed] (\x, 0) -- (\x, 3);
}
\node at (0.6, 0.6) {\blue{$\times$}};
\node at (0.2, 2.3) {\blue{$\times$}};
\node at (1.7, 0.8)  {\blue{$\times$}};
\node at (1.9, 0.2) {\blue{$\times$}};
\node at (1.3, 1.5) {\red{$\times$}};
\node at (2.5, 1.2) {\red{$\times$}};
\node at (2.7, 2.1) {\olive{$\times$}};
\end{scope}

\end{tikzpicture}
\caption{Left panel: We sample a Poisson point process with density $1$ and assign the nucleations different colors using the aforementioned rules. Although this happens on the entire first quadrant, we just show a picture of the square $[0, 3] \times [0, 3]$. Right panel: We sample the step colored $t$-PNG model in $[0,3] \times [0, 3]$ using these Poisson nucleations. Note that olive has a higher priority than red, and red has a higher priority than blue. Hence, the behavior of the olive lines does not depend on that of the red and blue lines. The behavior of the red lines does not depend on that of the blue lines.}
\label{fig:multi_color}
\end{figure}
 
We proceed to define the random variables $\{X_{m, n}, m, n \in \Z_{\geq 0}, m \leq n\}$. 
Let $H_{m, n}$ be the set of intersection points between the horizontal segment $[m, n] \times \{n\}$ and vertical lines in the step colored $t$-PNG model, where in the case of multiple lines traveling together, we count the intersection point just once. Let $v^{[m, n]}_z$ denote the number of vertical lines that go through $z$ with colors in $\{-n, \dots, -m-1\}$. We define    
\begin{equation}\label{eq:Xmn}
X_{m, n} = \sum_{z \in H_{m, n}} \big(v_z^{[m, n]} \text{ mod } 2\big).
\end{equation}
In other words, $X_{m, n}$ is the number of vertical lines with color $-n$ that cross the segment $[m, n] \times \{n\}$ after we recolor all lines with colors in $\{-n, \dots, -m-1\}$ with the color $-n$ and apply the mod $2$ erasure procedure.
\begin{proposition}\label{prop:equalind}
We have $\{X_{0, k}, k \in \Z_{\geq 1}\} = \{N(k, k), k \in \Z_{\geq 1}\}$ in distribution. 
\end{proposition}
\begin{proof}
It suffices to prove that for each $n \in \N$, 
$\{X_{0, k}, k = 1 \dots, n\} \overset{d}{=} \{N(k, k), k = 1, \dots, n\}$.
The lines in the square $[0, n] \times [0, n]$ have colors that belong to $\{-n, \dots, -1\}$. Replace 
these colors with a single color and apply the mod $2$ erasure procedure. 
Taking $m=1$ in Proposition \ref{prop:mod2erasure}, we see that the resulting model is just the 
single-colored
$t$-PNG model.
Note that $X_{0, k}$ in \eqref{eq:Xmn} is exactly the number of lines crossing the segment $[0, k] \times \{k\}$, which equals $N(k, k)$ for $k \in \{1, \dots, n\}$. This concludes the proposition. 
\end{proof}
\begin{proposition}\label{prop:item24}
The stochastic process $\{X_{m, n}: m, n \in \mathbb{Z}_{\geq 0}, m \leq n\}$ satisfies conditions \ref{item:ergodic}-\ref{item:expectation} of Theorem \ref{thm:liggettergodic}.
\end{proposition}
\begin{proof}
We first prove \ref{item:ergodic}. Consider the square $[m, n] \times [m, n]$. There are lines flowing inside through the left boundary $\{m\} \times [m, n]$ and the bottom boundary $[m, n] \times \{m\}$. These lines have colors that belong to $\{-m, \dots, -1\}$. The Poisson nucleations in $[m, n] \times [m, n]$ have colors in $\{-n, \dots, -m-1\}$. Note that the color $i$ takes priority over $j$ if $i < j$, so the lines that emanate from the Poisson nucleations in $[m, n] \times [m, n]$ have higher priority than the lines entering through the left and bottom boundaries. Therefore, by Proposition \ref{prop:colorignore}, we can sample the colored $t$-PNG model in the square $[m,n] \times [m,n]$ just using the colors $\{-n, \dots, -m-1\}$ and ignore the lower priority lines entering from the left and bottom. Hence, the distribution of $X_{m, n}$ is independent of the number and location of the lines entering the bottom and left boundaries of $[m, n] \times [m, n]$. This implies that for each $k \geq 1$, the random variables $\{X_{(n-1)k, nk}, n \geq 1\}$ are independent. It is straightforward to see that $X_{(n-1)k, nk}$ has the same distribution as $N(k, k)$ for all $n \geq 1$, so therefore this sequence is i.i.d and hence ergodic.

We proceed to prove \ref{item:equalind}. It suffices to show that for arbitrary $m \in \Z_{\geq 0}$, 
\begin{equation}\label{eq:equalind}
\{X_{m, m+k}, k \geq 0\} \overset{d}{=} \{X_{0, k}, k \geq 0\}.
\end{equation}
We look at the step colored $t$-PNG model restricted to 
$[m, \infty) \times [m, \infty)$. Note that there are lines with colors in $\{-m, \dots, -1\}$ entering from the left and bottom boundaries of $[m, \infty) \times [m, \infty)$. By Proposition \ref{prop:colorignore}, the behavior of lines in $[m, \infty) \times [m, \infty)$ with colors less than $-m$  is unaffected by the lower priority lines entering from the boundary. 
This implies that after a diagonal shift by $(m, m)$, the lines with colors $i_1, \dots, i_k \in \Z_{\leq -m-1}$ in $[m, \infty) \times [m, \infty)$ behave the same (in distribution) as the lines with colors $i_1 + m, \dots, i_k + m$ 
in $[0, \infty) \times [0, \infty)$. Hence, we conclude \eqref{eq:equalind}.  

Finally, since $X_{0, 1}$ is non-negative, $X_{0, 1}^{-} = 0$. Hence, \ref{item:expectation} holds. 
\end{proof}

Let us proceed to prove that $\{X_{m, n}, m, n \in \Z_{\geq 0}, m \leq n\}$ also satisfies the superadditive condition  \ref{item:subadditive} in Theorem \ref{thm:liggettergodic}. 
We begin with some preparation. In the square $[0, n] \times [0, n]$, we replace the colors $\{-m, \dots, -1\}$ with the color $-1$ and replace the colors $\{-n, \dots, -m-1\}$ with the color $-2$. After that, as long as there are two lines with the same color that travel together, we erase them.
By Proposition \ref{prop:mod2erasure}, the resulting model is a two-colored $t$-PNG model. In particular, the Poisson nucleations have color $-1$ in the L-shaped area $[0, m] \times [0, n] \cup [0, n] \times [0, m]$. The nucleation points have color $-2$ in the square $[m, n] \times [m, n]$.   
\begin{figure}[t]
\centering 
\begin{tikzpicture}
\draw (0, 0) -- (5, 0) -- (5, 5) -- (0, 5) -- (0, 0);
\draw[dashed] (3, 0) -- (3, 5);
\draw[dashed] (0, 3) -- (5, 3);
\node at (1, 2.1) {\blue{$\times$}};
\node at (2, 1) {\blue{$\times$}};
\node at (4, 1.5) {\blue{$\times$}}; 
\node at (2.5, 3.7) {\blue{$\times$}}; 
\node at (1.5, 4.3) {\blue{$\times$}};
\node at (3.3, 3.3) {\red{$\times$}};
\node at (4.5, 3.7) {\red{$\times$}};
\draw[thick, blue] (1, 2.1) -- (1, 5);
\draw[thick, blue] (1.5, 4.3) -- (1.5, 5);
\draw[thick, blue] (5, 1) -- (2, 1) -- (2, 5);
\draw[thick, blue] (1, 2.1) -- (4, 2.1) -- (4, 1.5) -- (5, 1.5);
\draw[thick, blue] (1.5, 4.3) -- (2.5, 4.3) -- (2.5, 3.7) -- (3.26, 3.7) -- (3.26, 5);
\draw[thick, red] (5, 3.3) -- (3.3, 3.3) -- (3.3, 5);
\draw[thick, red] (4.5, 5) -- (4.5, 3.7) -- (5, 3.7);
\node at (-0.3, 3) {$m$};
\node at (3, -0.3) {$m$};
\node at (-0.3, 5) {$n$};
\node at (5, -0.3) {$n$};
\end{tikzpicture}
\caption{We provide a possible sampling of the two-colored $t$-PNG model in the square $[0, m+n] \times [0, m+n]$.
The dashed lines $x = m$ and $y = m$ divide the square $[0, n] \times [0, n]$ into four rectangles. Blue represents the color $-1$, and red represents the color $-2$. In this example, we have $Q_1 = 2$ and $Q_2 = Q_{1,2 } = P_1 =  1$.}
\end{figure}

For the resulting two-colored $t$-PNG  model, let $Q_1$ be the number of be lines with color $-1$ that cross $[0, m] \times \{m\}$, let $Q_2$ be the number of vertical lines with color $-2$ that cross the segment $[m, n] \times \{n\}$, and let $P_1$ be the number of horizontal lines with color $-1$ that cross $\{m\} \times [m, n]$. Finally, let $Q_{1, 2}$ be the number of pairs of vertical lines of colors $-1$ and $-2$ that travel together and cross $[m, n] \times \{n\}$.

Consider the single-colored $t$-PNG model. For each Poisson nucleation or intersection point, the number of lines going upward or leftward equals the number of lines going downward or rightward (see Figure \ref{fig:L-matrix}). The next lemma follows immediately. 
\begin{lemma}\label{lem:lineconserve}
Consider the (single-colored) $t$-PNG model. Fix an arbitrary rectangle. The number of lines that cross the top and left boundaries of the rectangle equals the number of lines that cross the bottom and right boundaries.  
\end{lemma}
\begin{lemma}\label{lem:someequality}
The following result holds:
\begin{align}
\label{eq:u0m}
&
X_{0, m} = Q_1,
\\
\label{eq:umn}
&X_{m, n}= Q_2,  \\
\label{eq:u0n}
&X_{0, n} \geq Q_1 + P_1 + Q_2 - Q_{1, 2}.
\end{align} 
\end{lemma}
\begin{proof}
Recall that we obtain the two-colored $t$-PNG  model in $[0, n] \times [0, n]$ by replacing the colors $\{-n, \dots, -m-1\}$ with the color $-2$, replacing the colors $\{-m, \dots, -1\}$ with the color $-1$ and erasing pairs of lines with the same color. The erasure corresponds to the mod 2 erasure procedure in \eqref{eq:Xmn}. Hence, 
equations \eqref{eq:u0m} and \eqref{eq:umn} directly follow from \eqref{eq:Xmn}.

We proceed to prove \eqref{eq:u0n}. Note that $X_{0, n}$ is the number of single vertical lines (i.e. the line does not travel in a pair) that cross the segment  $[0, n] \times \{n\}$ in the two-colored $t$-PNG  model. We decompose 
\begin{equation}\label{eq:Xdecomp}
X_{0, n} = Y_1 + Y_2,
\end{equation}
where $Y_1$
equals the number of vertical lines with the color $-1$ that cross $[0, m] \times \{n\}$ and 
$Y_2$ equals the number of single vertical lines of either color 
that cross $(m, n] \times \{n\}$, excluding pairs. Note that in the rectangle $[0, m] \times [m, n]$, there are only lines with color $-1$. 
Applying Lemma \ref{lem:lineconserve} to the rectangle $[0, m] \times [m, n]$, we have
$Y_1 = Q_1 + P_1$. By definition, $Y_2 \geq Q_2 - Q_{1, 2}$. Using this together with \eqref{eq:Xdecomp}, we conclude \eqref{eq:u0n}.
\end{proof}

\begin{proposition}\label{prop:item1}
We have	$X_{0, 0} = 0$ and $X_{0, n} \geq X_{0, m} + X_{m, n}$ for $0 \leq m \leq n$. Hence, $\{X_{m, n}, m \leq n \in \Z_{\geq 0}\}$ satisfies \ref{item:subadditive} of Theorem \ref{thm:liggettergodic}.
\end{proposition}
\begin{proof}
By definition, we have $X_{0, 0} = 0$. We proceed to show that $X_{0, n} \geq X_{0, m} + X_{m, n}$. By Lemma \ref{lem:someequality}, it suffices to show that $P_1 \geq Q_{1, 2}$. We restrict ourselves to the square $[m, n] \times [m, n]$. All lines entering this square from the bottom and left boundaries have color $-1$, and all Poisson nucleations inside the square have color $-2$. $P_1$ equals the number of color $-1$ lines that enter the left boundary.
We can think of $Q_{1, 2}$ as the number of color $-2$ lines that cross $[m, n] \times \{n\}$ and are erased by a color $-1$ line. By looking at the possible two color configurations in Appendix \ref{sec:twocolorfig}, we find that each color $-1$ line which erases a color $-2$ line and crosses $[m, n] \times \{n\}$ must enter $[m, n] \times [m, n]$  from the left boundary. This implies that $P_1 \geq Q_{1, 2}$ and concludes the proposition. 
\end{proof}

The following scaling property follows immediately from the corresponding scaling property of the Poisson nucleations.
\begin{lemma}[Scaling]\label{lem:scaling}
For a fixed $0 < s < \infty$, we have $$(N(x, y); x, y \geq 0) \overset{d}{=} (N(sx,y/s); x, y \geq 0).$$
\end{lemma}
The following proposition partially proves Theorem \ref{hydrodynamic_limit}.
\begin{proposition}\label{prop:LLN}
Let $\gamma = \sup_{n\geq 1} \frac{\mathbb{E}[N(n, n)]}{n}$.
With probability $1$, we have 
$$\lim_{s \to \infty} \frac{N(sx, sy)}{s} = \gamma \sqrt{xy}$$ for any $x, y > 0$.
\end{proposition}
\begin{proof}
We first prove that almost surely, $\frac{N(s, s)}{s} \to \gamma$. 
Using Proposition \ref{prop:item24} and Proposition \ref{prop:item1}, we can now apply Theorem \ref{thm:liggettergodic} to conclude that almost surely $\frac{X_{0, n}}{n} \to \gamma$ as $n \to \infty$, where $\gamma = \sup_{n\geq 1} \frac{\mathbb{E}[X_{0, n}]}{n}$. The convergence also holds in $L^1$ if $\gamma$ is finite. Using Proposition \ref{prop:equalind}, we have that $\lim_{n \to \infty}\frac{N(n, n)}{n} = \gamma$ almost surely. Note that since $N(s, s)$ is increasing in $s$, we also have that almost surely,
\begin{equation*}
\lim_{t \to \infty} \frac{N(s, s)}{s} = \gamma.
\end{equation*}
This together with  Lemma \ref{lem:scaling} 
implies that almost surely $\lim_{s \to \infty} \frac{N(sx, sy)}{s} = \gamma \sqrt{xy}$ for arbitrary fixed $x, y > 0$. We take a probability $1$ event such that $\lim_{s \to \infty} \frac{N(sx, sy)}{s} = \gamma \sqrt{xy}$ for $x, y \in \mathbb{Q}_{>0}$. By the density of $\mathbb{Q}_{>0}$ in $\R_{>0}$ and the monotonicity of the height function, we know that on that 
event, $\lim_{s \to \infty} \frac{N(sx, sy)}{s} = \gamma\sqrt{xy}$ for all $x, y \in \R_{> 0}$. 
\end{proof}

To complete the proof of Theorem \ref{hydrodynamic_limit}, it remains to identify the constant $\gamma$. This will be done in Section \ref{sec:proving}. Before carrying out the proof, we need to first show that $\gamma$ is neither zero nor infinity. The fact that $\gamma$ is non-zero can be easily seen by the following lemma.
\begin{lemma}\label{lem:gamma2}
We have $\gamma \geq 2$. 
\end{lemma}
\begin{proof}	
Let $N_0 (x, y)$ be the height function of the PNG model (where $t = 0$). We can couple together the $t$-PNG model and the PNG model so that they have the same Poisson nucleations. Under this coupling, we always have $N(x, y) \geq N_{0} (x, y)$. Using this together with the law of large numbers for the PNG model, 
we conclude that $\gamma \geq 2$.
\end{proof} 


\section{Stationary model and upper bound} \label{sec:upper-bound}
In this section, we prove that $\gamma$ is finite by constructing a stationary version of the $t$-PNG model and comparing it with the original $t$-PNG model. 
We remark that the existence of a stationary version of the $t$-PNG model is not surprising, since similar stationary models have been observed for the last passage models \cite{HammersleySourcesAndSinks, cator2006second, seppalainen2017variational}, polymers \cite{o2001brownian, seppalainen2012scaling, chaumont2018characterizing}, and stochastic vertex models \cite{aggarwal2018current, imamura2020stationary, lin2020kpz}. 
The results in this section can be viewed as generalizations of the results in Section \ref{sec:subadditivity} of \cite{HammersleySourcesAndSinks}.

Fix $\lambda, T_1, T_2 > 0$.
Consider the $t$-PNG model on $[0,T_1] \times [0, T_2]$ with the following boundary data: a Poisson point process of sources of intensity $\lambda$ on the bottom boundary and a Poisson point process of sinks of intensity $\frac{1}{\lambda(1-t)}$ on the left boundary. In order to study the stationarity of this model, it is helpful to understand the model as an interacting particle system. 

Given the $t$-PNG model, we can naturally obtain an interacting particle system called the 
\emph{$t$-Hammersley process} as follows: Let $(X_{\tau})_{0\leq \tau \leq T_2}$ be the configuration of particle locations in $[0, T_1]$ at time $\tau$. To avoid ambiguity, we let $X_{\tau}$ be right continuous. 
The particle locations in this interpretation are the locations $x$ such that $(x,\tau)$ belongs to a vertical line segment in the $t$-PNG model. More precisely, 
$X_{\tau}$ is a Markov chain on the state space $E$ consisting of all finite point configurations on $[0, T_1]$. We can decompose $E =  \bigsqcup_{n=0}^{\infty}E_n$ where each $E_n$ consists of particle configurations with exactly $n$ points: 
$$E_n = \{(x_1, \ldots, x_n): 0 \leq x_1 \leq \ldots \leq x_n \leq T_1\} \quad \text{when }  (n\geq 1),$$
and $E_ 0 = \{\emptyset\}$, where $\emptyset$ is the empty configuration. We give each set $E_n$ the usual 
topology so that $E$ is a locally compact space. The reason that we allow multiple points at the same location is purely technical; we want $E$ to be a Polish space if we identify the point configurations with Radon measures.

Due to our choice of boundary data, $X_0$ is a Poisson process with intensity $\lambda$. Using the infinitesimal generator of $X$, we will prove that the Markov process $X$ is stationary, meaning that the point configuration $X_{\tau}$ will remain a Poisson point process with intensity $\lambda$ for all $\tau \in [0, T_2]$.

We define the infinitesimal generator of $X$. Let us first define  
two families of operators: $(\mathcal{R}_z^i)_{i=1}^{\infty}$ and $(\mathcal{L}^i)_{i=1}^{\infty}$. For each $z \in (0, T_1)$ and $i \geq 1$ we can define the operator $\mathcal{R}^i_{z}: E \rightarrow E$ such that for $x_{m-1} < z \leq x_m$ (take $x_0 = 0$ and $x_{n+1} = \infty$), 
\begin{equation}
\mathcal{R}^i_{z} x= \begin{cases}(x_{1}, \ldots, x_{m-1}, z , x_{m} ,\ldots,\widehat{x}_{m +i-1}, \ldots, x_{n}), &  \text{if } x \in E_n, i \leq n-m+1, \\ (x_{1}, \ldots ,x_{m-1}, z , x_{m}, \ldots, x_{n}),& \text{if }x \in E_n, i > n-m+1, \end{cases}
\end{equation}
where $\widehat{x}_k$ denotes that the particle at $x_k$ is removed from the configuration. 
The operator $\mathcal{R}_z^i$ describes what happens when there is a nucleation at position $z$. It inserts a particle at $z$ and removes the particle whose position is $i$ positions to the right of $z$ if there are at least $i$ particles to the right of $z$. We can also think of $\mathcal{R}_z^i$ as sequentially shifting over the $i$ particles at positions $x_{m}, \dots, x_{m+i-1}$ to positions $z, x_m, \dots, x_{m+i-2}$. If there are not $i$ particles to shift over, then we shift over $x_m, \dots, x_n$ and insert a new particle ``from infinity" at $x_n$, as illustrated in Figure \ref{fig:particle}. 
The advantage of this viewpoint is that it maintains the ordering between the particles. 

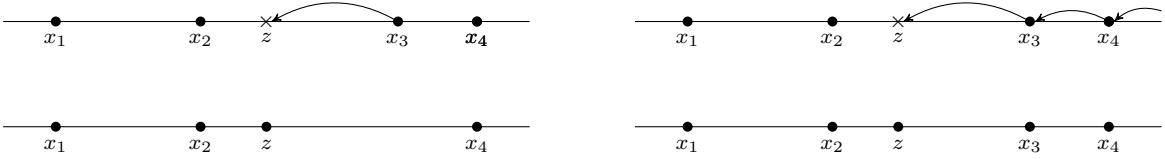
\begin{figure}[t] 
\centering
\begin{tikzpicture}[scale=0.7]
\draw (0, 0) -- (10, 0); 
\node at (5, 0) {$\times$}; 
\node[label=below:{$x_1$}, fill=black, circle, scale=0.5] at (1,0) {};
\node[label=below:{$x_2$}, fill=black, circle, scale=0.5] at (3.75,0) {};
\node[label=below:{$z$}, scale = 0.5] (1) at (5,0) {};
\node[label=below:{$x_3$}, fill=black,
circle, scale=0.5] at (7.5,0) {};
\node[label=below:{$x_4$}, fill=black, circle, scale=0.5](2)  at (9,0) {};
\node[label=below:{$x_4$}, fill=black, circle, scale=0.5](2)  at (9,0) {};
\path
(7.5, 0) edge[bend right, ->,>=stealth'] (5.1, 0);

\begin{scope}[yshift=-2cm]
\draw (0, 0) -- (10, 0); 
\node[label=below:{$x_1$}, fill=black, circle, scale=0.5] at (1,0) {};
\node[label=below:{$x_2$}, fill=black, circle, scale=0.5] at (3.75,0) {};
\node[label=below:{$z$}, fill=black, circle, scale=0.5] at (5,0) {};
\node[label=below:{$x_4$}, fill=black,
circle, scale=0.5] at (9,0) {};
\end{scope}

\begin{scope}[xshift= 12cm]
\draw (0, 0) -- (10, 0); 
\node at (5, 0) {$\times$};
\node[label=below:{$x_1$}, fill=black, circle, scale=0.5] at (1,0) {};
\node[label=below:{$x_2$}, fill=black, circle, scale=0.5] at (3.75,0) {};
\node[label=below:{$z$}, scale = 0.5] (1) at (5,0) {};
\node[label=below:{$x_3$}, fill=black,
circle, scale=0.5] at (7.5,0) {};
\node[label=below:{$x_4$}, fill=black, circle, scale=0.5](2)  at (9,0) {};
\node[fill=black, circle, scale=0.5](2)  at (9,0) {};
\path
(9, 0) edge[bend right, ->,>=stealth'] (7.6, 0);
\node[fill=black, circle, scale=0.5](2)  at (9,0) {};
\path
(7.5, 0) edge[bend right, ->,>=stealth'] (5.1, 0);
\path
(10, 0.2) edge[bend right, ->,>=stealth'] (9.1, 0);
\end{scope}

\begin{scope}[yshift=-2cm, xshift= 12cm]
\draw (0, 0) -- (10, 0); 
\node[label=below:{$x_1$}, fill=black, circle, scale=0.5] at (1,0) {};
\node[label=below:{$x_2$}, fill=black, circle, scale=0.5] at (3.75,0) {};
\node[label=below:{$z$}, fill=black,
circle, scale=0.5] at (5,0) {};
\node[label=below:{$x_3$}, fill=black,
circle, scale=0.5] at (7.5,0) {};
\node[label=below:{$x_4$}, fill=black, circle, scale=0.5](2)  at (9,0) {};

\end{scope}

\end{tikzpicture}
\caption{Example of $\mathcal{R}_z^ix$ with $n=4$ and $m=3$. On the left we take $i=1$ and on the right we take $i > 2$. Note that for $i \leq  n-m+1 = 2$, the number of particles is preserved and for $i >2$, the number of particles increases by $1$.}
\label{fig:particle}
\end{figure}

We also define $\mathcal{L}^i: E \rightarrow E$ such that 
\begin{equation}
\mathcal{L}^i x= \begin{cases}\left(x_1, \ldots, \widehat{x_i}, \ldots,  x_{n}\right), & \text { if  }x \in E_n , i \leq n, \\ (x_1,\ldots, x_n), & \text { if } x \in E_n, i > n.\end{cases}
\end{equation}
The operator $\mathcal{L}^i$ describes what happens when there is a sink on the left boundary. It removes the $i$th particle if there is one, and if not, does nothing. 

\begin{proposition}\label{generator-X}
Let $X_{\tau}$ be the $t$-Hammersley process for $t \in [0,1)$ with Poisson sources of intensity $\lambda$ and Poisson sinks of intensity $\frac{1}{\lambda(1-t)}$. Its generator $G$ is given by the following formula when acting on $f \in C_{0}(E)$:
\begin{equation}
    Gf(x) = \sum_{i=1}^{\infty}t^{i-1}(1-t)\int_0^{T_1}(f(\mathcal{R}^i_z x) - f(x))dz + \sum_{i=1}^{\infty}\frac{t^{i-1}}{\lambda}(f(\mathcal{L}^ix) - f(x)).
\end{equation}
\end{proposition}

\begin{proof}
The first term on the right comes from moving the particle configuration from $x$ to $\mathcal{R}_z^i x$, where $i$ is sampled from a geometric distribution and $z$ is chosen according to the uniform measure on $[0, T_1]$. The second term comes from moving the particle configuration from $x$ to $\mathcal{L}^i x$ for some $i$ ($i$ is again sampled from a geometric distribution) with rate $\frac{1}{\lambda(1-t)}$.  
A rigorous proof follows from a direct computation as in \cite{HammersleySourcesAndSinks}.
\end{proof}

Let $G^*$ be the dual of $G$ with respect to $\mu$. In other words, the operator satisfying
\begin{equation}
    \int_E Gf(x)g(x)\mu(dx) = \int_Ef(y)G^*g(y)\mu(dy) \ \ \text{for all $f,g \in C_0(E)$.}
\end{equation}

To compute $G^*$, we will need to define two additional sets of operators.
These operators are similar to the ones above except that they move particles from left to right instead of right to left. 
For each $s \in (0, T_1)$ and $i \geq 1$ we can define the operator $\mathcal{L}^i_{s}: E \rightarrow E$ such that for $x_{m-1} < s \leq x_m$,
\begin{equation}
 \mathcal{L}^i_{s} x= \begin{cases}\left(x_{1}, \ldots, \widehat{x}_{m -i}, \ldots, x_{m-1},s, x_{m} ,\ldots, x_{n}\right), &  
 \text{if }x \in E_n, i \leq m-1,
 \\ (x_{1}, \ldots ,x_{m-1},s, x_{m}, \ldots, x_{n}),& 
 \text{if }x \in E_n, i > m-1.
 \end{cases}
\end{equation}
This operator inserts a particle at position $s$ and removes the particle whose position is $i$ positions to the left of $s$ if there are at least $i$ particles to the left of $s$.

We also define $\mathcal{R}^i: E \rightarrow E$ such that 
\begin{equation}
\mathcal{R}^i x= \begin{cases}\left(x_1, \ldots, \widehat{x}_{n-i+1}, \ldots,  x_{n}\right), & \text { if  }x \in E_n , i \leq n, \\ (x_1,\ldots, x_n), & \text { if } x \in E_n, i > n.\end{cases}
\end{equation}
This operator removes the $(n-i+1)$th 
particle if there is one, and if not, does nothing. 
\begin{lemma}
For all $g \in C_{0}(E)$,
\begin{equation}
G^*g(y) = \sum_{i=1}^{\infty}t^{i-1}(1-t)\int_0^{T_1}(g(\mathcal{L}^i_sy) - g(x))ds + \sum_{i=1}^{\infty}\frac{t^{i-1}}{\lambda}(g(\mathcal{R}^iy) - g(y)).
\end{equation}
\end{lemma}
\begin{proof}
The proof is similar to \cite{HammersleySourcesAndSinks}; however, since the dynamics of the $t$-Hammersley process are more intricate for $t > 0$, the computations are more complicated. We record the details in Lemma \ref{lem:dualGen} in Appendix \ref{sec:generators}.
\end{proof}
We can see that $G^*1 = 0$. This proves that $\mu$ is stationary for the Markov process $X$. Using the particle system interpretation, we 
now prove Burke's theorem for the $t$-PNG model. 


\begin{proof}[Proof of Theorem \ref{thm:burke}]
The argument intrinsically follows Theorem 3.1 in \cite{HammersleySourcesAndSinks}. We define the time-reversed Markov process $(\widetilde{X}_\tau)_{0 \leq \tau \leq T_2}$ by $$\widetilde{X}_\tau = \lim_{\tau' \downarrow \tau} X_{T_2 - \tau'}.$$ We define this process using left limits to make sure that it is c\`adl\`ag. 
Since $G^*$ and $G$ are dual with respect to the stationary measure $\mu$, it is standard that 
$G^*$ is in fact the generator of the time-reversed process $\widetilde{X}_s$. 

However, by a similar argument to Proposition \ref{generator-X}, we see that $G^*$ is also the generator of the process $X_V$ obtained by a vertical reflection of all the space-time paths of $L_{\lambda}$ across the vertical line segment $\{\frac12 T_1\} \times [0, T_2]$. Particles in $X$ jump from right to left, while in $X_V$ we reverse the direction and particles jump from left to right. 

The time-reversed process $\widetilde{X}$ can be obtained from the original space-time paths of $L_{\lambda}$ by performing a horizontal reflection across the horizontal line segment $[0,T_1] \times \{\frac12 T_2\}.$ Accordingly, we can rename this process $X_H = \widetilde{X}$. The two reflected processes $X_V$ and $X_H$ share the same generator $G^*$. In addition, both processes start with Poisson initial data with intensity $\lambda$, so we can conclude that the two processes are equal in distribution.

\begin{figure}[t] 
\centering
\begin{tikzpicture}[scale=0.6]
\draw (0, 0) -- (6, 0); 
\draw (0, 0) -- (0, 6);
\node at (3, 3) {$\times$};
\node at (2, 2) {$\times$};
\draw (0, 6) -- (6, 6) -- (6, 0);
\draw [thick](0,5) -- (1, 5) -- (2, 5) -- (2, 2) -- (4, 2) --(4,0);
\draw[thick] (2, 6) -- (2, 5) -- (3, 5) -- (3, 3) -- (6, 3);
\draw[thick] (5, 0) -- (5, 6);
\draw[fill, red] (2, 5) circle (0.1);
\draw[fill, red] (5, 3) circle (0.1);
\draw[fill, blue] (3,5) circle (0.1);
\draw[fill, blue] (4,2) circle (0.1);

\begin{scope}[xshift = 14cm, xscale=-1]
\draw (0, 0) -- (6, 0); 
\draw (0, 0) -- (0, 6);
\node at (3, 3) {$\times$};
\node at (2, 2) {$\times$};
\draw (0, 6) -- (6, 6) -- (6, 0);
\draw [thick](0,5) -- (1, 5) -- (2, 5) -- (2, 2) -- (4, 2) --(4,0);
\draw[thick] (2, 6) -- (2, 5) -- (3, 5) -- (3, 3) -- (6, 3);
\draw[thick] (5, 0) -- (5, 6);
\draw[fill, red] (2, 5) circle (0.1);
\draw[fill, red] (5, 3) circle (0.1);
\draw[fill, blue] (3,5) circle (0.1);
\draw[fill, blue] (4,2) circle (0.1);
\end{scope}

\begin{scope}[xshift = 16cm, yshift=6cm, yscale=-1]
\draw (0, 0) -- (6, 0); 
\draw (0, 0) -- (0, 6);
\node at (3, 3) {$\times$};
\node at (2, 2) {$\times$};
\draw (0, 6) -- (6, 6) -- (6, 0);
\draw [thick](0,5) -- (1, 5) -- (2, 5) -- (2, 2) -- (4, 2) --(4,0);
\draw[thick] (2, 6) -- (2, 5) -- (3, 5) -- (3, 3) -- (6, 3);
\draw[thick] (5, 0) -- (5, 6);
\draw[fill, red] (2, 5) circle (0.1);
\draw[fill, red] (5, 3) circle (0.1);
\draw[fill, blue] (3,5) circle (0.1);
\draw[fill, blue] (4,2) circle (0.1);
\end{scope}

\end{tikzpicture}
\caption{From left to right these are the processes $X$, $X_V$ and $X_H = \widetilde{X}$. The latter two figures are obtained from the first by a vertical and horizontal reflection, respectively. The second and third pictures are equal in distribution, hence the blue circles (corner points of $X$) have the same distribution as the $\times$'s (nucleations of $X$).}
\label{fig:Burke}
\end{figure}
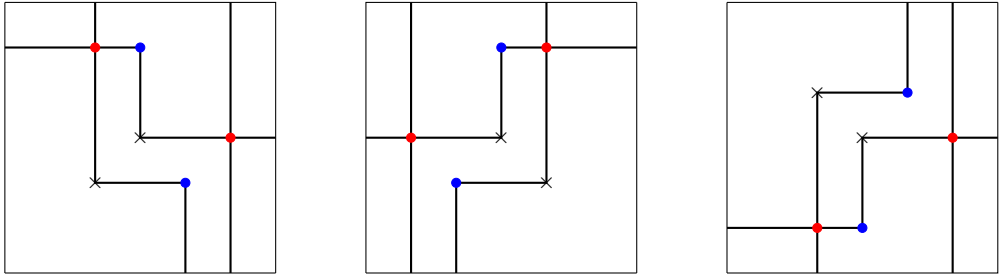

In $X_V$, particles are inserted at the vertical reflection of the Poisson nucleations of the original process. In $X_H$, particles are inserted at the horizontal reflection of the corner points of the original process (see Figure \ref{fig:Burke}). These two sets must be equal in distribution and since the Poisson process is invariant under reflections, this shows that the distribution of the corner points of $L_{\lambda}$ is also a Poisson point process with intensity $1$. 

In the process $X_V$, paths exit on the right side according to a Poisson process with intensity $\frac{1}{\lambda(1-t)}$. In $X_H$, paths exit on the right side according to the horizontal reflection of $L_{\lambda}^{\text{in}}.$ Hence $L_{\lambda}^{\text{in}}$ is a Poisson process with intensity $\frac{1}{\lambda(1-t)}.$ Similarly, for $X_V$, paths enter from the bottom with the vertical reflection of a Poisson process of intensity $\lambda$ and this is equal in distribution to the paths entering $X_H$, but those are just $L_{\lambda}^{\text{out}}$. the independence of $L_{\lambda}^{\text{in}}, L_{\lambda}^{\text{out}}$ and $L_{\lambda}^{\text{corner}}$ follows from the fact that these three processes are independent for $X_V$ since they are equal to the sources, sinks and nucleations of $L_{\lambda}.$
\end{proof}
Using the stationarity of $X$, we can now compute the upper bound for $\gamma$. Note that we only use stationarity and do not need Burke's theorem. We will compare the $t$-PNG model with empty boundary data to the stationary model $L_{\lambda}$ defined above. To this end, we 
need the following lemma showing that the $t$-PNG model is attractive as an interacting particle system (see \cite{liggett2012interacting}). It follows that the height function of the $t$-PNG model is monotone under adding sources and sinks.

\begin{lemma}[Attractivity] \label{lem:attractivity} 

We have 
\begin{enumerate}[leftmargin = 2em, label = (\roman*)] 
\item Let $\eta_{\tau}$ and $\zeta_{\tau}$ denote two $t$-Hammersley processes on the rectangle $[0, T_1] \times [0, T_2]$ with sinks given by $\xi$ and sources given by $\eta_0 \subseteq \zeta_0$. Then there exists a coupling of the two
processes $\eta_{\tau}$ and $\zeta_{\tau}$ such that 
$\eta_{\tau} \subseteq \zeta_{\tau}$ 
for all $\tau \in [0, T_2]$.
\item Let $\eta_{\tau}$ and $\tilde{\eta}_{\tau}$ denote two $t$-Hammersley processes on the rectangle $[0, T_1] \times [0, T_2]$ with sinks given by $\xi \subseteq \widetilde{\xi}$, respectively, and sources given by $\eta_0$. Then there exists a coupling of the $t$-PNG 
processes $\eta_{\tau}$ and $\widetilde{\eta}_{\tau}$ such that for all $\tau \in [0, T_2]$, the number of particles in $\tilde{\eta}_{\tau}$ must be at least the number of particles in $\eta_{\tau}$ minus the number of sinks in $\tilde{\xi} \setminus \xi$ added up until time $\tau$.  
\item We conclude that if we add either sources or sinks to the $t$-PNG model then there is a coupling so that the height function weakly increases.

\end{enumerate}
 
\end{lemma}
\begin{proof}

\begin{enumerate}[leftmargin = 2em, label = (\roman*)]
\item \label{add-sources} We color all the sources in $\eta_0$ with color $1$ and color the sources in $\zeta_0 \setminus \eta_0$ with color $2$. We sample $\eta_{\tau}$ by ignoring the second color and sample $\zeta_{\tau}$ by performing the mod 2 erasure procedure on both colors. Note that the sources of color $2$ can only erase lines in the horizontal direction and therefore cannot erase any vertical lines which represent particles in $\eta_{\tau}$ (see the possible two-colored configurations in Appendix \ref{sec:twocolorfig}). Therefore $\eta_{\tau} \subseteq \zeta_{\tau}$.

\item \label{add-sinks} Similarly, we color all the sinks in $\xi$ with color $1$ and color all the sinks in $\tilde{\xi} \setminus \xi$ with color $2$. We sample $\eta_{\tau}$ by ignoring the second color and sample $\tilde{\eta}_{\tau}$ by performing the mod 2 erasure procedure on both colors. Each sink of color $2$ can delete at most one vertical line, hence the number of particles decreases by at most the number of sinks in $\tilde{\xi} \setminus \xi$ added up until time $\tau$. 

\item Note that 
$$N(x, y) = \text{number of sinks in } \{0\} \times [0,y] + \text{number of particles in } [0, x] \times \{y\}.$$ It follows from \ref{add-sources} and \ref{add-sinks} that the height function does not decrease if we add sources and sinks using the above couplings. 

\end{enumerate}

\end{proof}

\begin{lemma}[Upper Bound]
We have 
\begin{equation}\label{upper-bound}
\gamma \leq \frac{2}{\sqrt{1-t}}. 
\end{equation}
\end{lemma}

\begin{proof}
Consider the $t$-PNG model without boundary data as well as the stationary $t$-PNG model $L_\lambda$ defined above. Let $N_\lambda^{\text{stat}}$ be the height function of the stationary model. 
We have
\begin{align*}
    \E[N(x,y)] &\leq \E[N^{\text{stat}}_{\lambda}(x,y)] 
    =  \E[N^{\text{stat}}_{\lambda}(0, y)] +  \E[N^{\text{stat}}_{\lambda}(x,y)-N^{\text{stat}}_{\lambda}(0, y)] 
    = \frac{y}{\lambda(1-t)} + \lambda x.
\end{align*}
The first inequality follows from Lemma \ref{lem:attractivity}. The third equality follows because $N^{\text{stat}}_{\lambda}(0,y)$ is just the number of points in a Poisson point process of intensity $\frac{1}{\lambda(1-t)}$ in an interval of length $x$, and $N^{\text{stat}}_{\lambda}(x,y)-N^{\text{stat}}_{\lambda}(0, y)$ is the number of points in a Poisson process of intensity $\lambda$ in an interval of length $y$ (due to the stationary of $X$). 
 Taking $\lambda = \frac{\sqrt{y}}{\sqrt{x(1-t)}}$ (which minimizes the right-hand side above), we conclude that 
\begin{equation}
\frac{ \E[N(x, y)] }{\sqrt{xy}}\leq \frac{2}{\sqrt{1-t}}.
\end{equation}
Recall that $\gamma = \sup_{n \geq 1} \frac{\mathbb{E}[N(n, n)]}{n}$. 
The above inequality implies that $\gamma \leq \frac{2}{\sqrt{1-t}}$.
\end{proof}


\section{Proving $\gamma = \frac{2}{\sqrt{1-t}}$} \label{sec:lower-bound}
\label{sec:proving}
In this section, we prove that $\gamma = \frac{2}{\sqrt{1-t}}.$ Recall that when two lines intersect, the intersection point is called a crossing point if the two lines cross and a corner point if the two lines annihilate. The set of $\alpha$-points is the union of the Poisson nucleations and crossing points, and the set of $\beta$-points is the union of corner points and crossing points.
For the $t$-PNG model with empty boundary data, let $\A_1(x, y)$ be the number
of $\alpha$-points and $\A_2 (x, y)$ be the number of $\beta$-points in the rectangle $[0, x]\times [0,y]$. As an abbreviation, we use $\A_1 (s)$ to denote $\A_1 (s, s)$.
\subsection{Law of large numbers for the $\alpha$-points}
The main result in this subsection is the following law of large numbers:
\begin{proposition}\label{prop:alphaptlim}
We have $\frac{1}{s^2} \A_1(s) \overset{p}{\to} \frac{1}{1-t}$, as $s \to \infty$.
\end{proposition}
We will need to prove a few lemmas before proving this proposition. Fix any rectangle $R = [x_1, x_2] \times [y_1, y_2]$. Define \begin{equation*}
\h_R = \min\Big\{\text{number of intersections between } [x_1, x_2] \times \{y\} \text{ and the vertical lines in $t$-PNG model}: y \in [y_1, y_2] \Big\}.
\end{equation*}
In other words, as we move the horizontal segment vertically from $[x_1, x_2] \times \{y_1\}$ to $[x_1, x_2] \times \{y_2\}$, $h_R$ is the minimum number of intersections between the vertical lines and this horizontal segment. 
\begin{lemma}\label{lem:vlinelowerbd}
Fix any rectangle $R = [x_1, x_2] \times [y_1, y_2]$. Let $\ell$ be the number of horizontal lines that enter the rectangle from the left. Let $b$ be the number of vertical lines that enter the rectangle $R$ from the bottom. Then $\h_R \geq (b-\ell)^+$.
\end{lemma}
\begin{proof}
In the $t$-PNG model, a vertical line will go upward until it meets a horizontal line. Moreover, each horizontal line can annihilate at most one vertical line. This concludes the lemma. 
\end{proof}

Fix $M \in \mathbb{Z}_{\geq 2}$. Let $R(s) = [(1- \frac{1}{M}) s, s] \times [\frac{s}{M}, s]$.
\begin{lemma}
We have almost surely,
	\begin{equation}\label{eq:nt}
	\lim_{s \to \infty} \h_{R(s)} \to \infty.
	\end{equation}
\end{lemma}
\begin{proof}
For each $i \in \{0, 1, \dots, M^2-1\}$, let $b_i(s)$ be the number of lines entering the rectangle $R_i = [(1- \frac{1}{M}) s, s] \times [\frac{is}{M^2}, \frac{(i+1)s}{M^2}]$ from the bottom boundary. 
	Let $\ell_i (s)$ be the number of lines entering the same rectangle from the left boundary. 
	It is straightforward that  
	\begin{align*}
	&b_i (s) = N\Big(s, \frac{is}{M^2}\Big) - N\Big((1-\frac{1}{M})s, \frac{is}{M^2}\Big),\\
	&\ell_i (s) = N\Big((1-\frac{1}{M})s, \frac{(i+1) s}{M^2}\Big) - N\Big((1-\frac{1}{M})s, \frac{i s}{M^2}\Big).
	\end{align*}
By Lemma \ref{lem:vlinelowerbd}, we know that $h_{R_i(s)} \geq b_i(s) - \ell_i (s)$. 
Since $R = \cup_{i = M}^{M^2-1} R_i$, we have $h_{R(s)} \geq \min_{i = M}^{M^2-1} \big(b_i (s) - \ell_i (s)\big)$.

	By the law of large numbers in Proposition \ref{prop:LLN}, we know that almost surely,
	\begin{align*}
	&\lim_{s \to \infty} \frac{b_i(s)}{s} = \gamma \frac{\sqrt{i}}{M}\Big(1 - \sqrt{1 - \frac{1}{M}}\Big),
	\\
	&\lim_{s \to \infty} \frac{\ell_i (s)}{s} = \gamma \sqrt{1 - \frac{1}{M}} \Big(\frac{\sqrt{i+1}}{M} - \frac{\sqrt{i}}{M}\Big).
	\end{align*}
Subtracting the second equation from the first above yields
	\begin{equation*}
	\lim_{s \to \infty} \frac{b_i (s) - \ell_i (s)}{s} = \gamma \bigg(\frac{\sqrt{i}}{M}\Big(1 - \sqrt{1 - \frac{1}{M}}\Big) - \sqrt{1 - \frac{1}{M}} \Big(\frac{\sqrt{i+1}}{M} - \frac{\sqrt{i}}{M}\Big)\bigg). 
	\end{equation*}  
	When $i \geq M$, it is straightforward to check that 
	\begin{equation*}
	\frac{\sqrt{i}}{M}\Big(1 - \sqrt{1 - \frac{1}{M}}\Big) - \sqrt{1 - \frac{1}{M}} \Big(\frac{\sqrt{i+1}}{M} - \frac{\sqrt{i}}{M}\Big) > 0.
	\end{equation*}
	This together with $\gamma > 0$ (see Lemma \ref{lem:gamma2}) implies that for any $i \geq M$, we have $\lim_{s \to \infty} \frac{b_i (s) - \ell_i (s)}{s} > 0$. Since 
	$\h_{R(s)} \geq \min_{i = M}^{M^2-1} \big( b_i (s) - \ell_i (s) \big)$, we conclude that $\lim_{s \to \infty} \frac{h_{R(s)}}{s} > 0$. This concludes the lemma. 
\end{proof} 
\begin{lemma}\label{lem:convergeexp}
We have 
\begin{equation*}
\lim_{s \to \infty} s^{-2} \mathbb{E}\big[\A_1(s)\big] = \frac{1}{1-t}.
\end{equation*} 
\end{lemma}
\begin{proof}
We prove the lemma by showing that 
\begin{align}
\label{eq:upbd}\tag{\textbf{upper bound}}
\limsup_{s \to \infty} s^{-2} \mathbb{E}[\A_1 (s)] \leq \frac{1}{1-t},\\
\label{eq:lwbd}\tag{\textbf{lower bound}}
\liminf_{s \to \infty} s^{-2} \mathbb{E}[\A_1 (s)] \geq \frac{1}{1-t}.   
\end{align}
\textbf{Proof of \eqref{eq:upbd}:}
Let $K(s)$ be the total number of Poisson nucleations in $[0, s] \times [0, s]$. We label these points by $1, \dots, K(s)$. Let $\alpha_i (s)$ be the number of vertical lines in the $t$-PNG model that 
cross the 
the horizontal line emanating from the nucleation point $i$ (including the vertical line that emanates from the point $i$). 
For every $n \in \Z_{\geq 1}$, 
\begin{equation}\label{eq:geoupbd}
\mathbb{E}\big[\alpha_n(s)\,|\, K(s) \geq n\big] \leq \frac{1}{1-t},
\end{equation} 
This is because if we condition on the event $\{K(s) \geq n\}$, $\alpha_n (s)$ is stochastically dominated by a geometric random variable with parameter $1-t$.
As a result, we have 
	\begin{equation*}
	\mathbb{E}\big[\A_1(s)\big] = \mathbb{E}\Big[\sum_{i = 1}^{K(s)} \alpha_i (s)\Big]  = \sum_{n = 1}^\infty \mathbb{P}(K(s) \geq n) \mathbb{E}\big[\alpha_n(s) \,|\, K(s) \geq n\big] \leq \frac{1}{1-t} \mathbb{E}\big[K(s)\big] = \frac{s^2}{1-t}. 
	\end{equation*}
For the first equality, note that there are $\alpha_i (s)$ number of $\alpha$-points that have the same $y$-coordinate as that of the nucleation point $i$, hence $\mathsf{A}_1 (s) = \sum_{i = 1}^{K(s)} \alpha_i (s)$.
The first inequality above is due to \eqref{eq:geoupbd}. The third equality holds because $K(s)$ is a Poisson random variable with parameter $s^2$.
	Hence, we have shown that 
	\begin{equation}\label{eq:upperbd}
	\limsup_{s \to 
		\infty} \frac{1}{s^2} \mathbb{E}\big[ \A_1 (s)\big] \leq \frac{1}{1-t}.
	\end{equation}
\textbf{Proof of \eqref{eq:lwbd}:}
Fix $M, m \in \Z_{\geq 2}$. Define $c_M(\e) = (1 - \frac{1}{M})^2 - \e$. We can choose $\e > 0$ small enough such that $c_M(\e) > 0$. Let $K_M(s)$ be the total number of Poisson nucleations in the square $[0, (1-\frac{1}{M})s)] \times [ \frac{s}{M}, s]$. We have 
\begin{align*}
\E\big[\A_1(s)\big] &\geq\mathbb{E}\Big[\sum_{i = 1}^{K_M (s)}  \alpha_i (s)\Big] \geq 
\mathbb{E}\Big[\Big(\sum_{i = 1}^{c_M(\e) s^2} \alpha_i (s)\Big) \mathbbm{1}_{\{K_M(s) \geq c_M(\e) s^2, \h_R (s) \geq m\}}\Big]. 
\end{align*}
For a line that emanates rightward from a Poisson nucleation in $[0, (1-\frac{1}{M})s)] \times [(1 - \frac{1}{M}) s, s]$, there are potentially at least $h_R (s)$ number of vertical lines it can cross. 
Hence, there exist i.i.d. geometric random variables $\{X_i\}_{i \in \Z_{\geq 1}} \sim \text{Geo}(1-t)$ that are independent of the $t$-PNG model, and such that $\alpha_i(s) \geq X_i \wedge \h_R(s)$. Therefore, we have
\begin{equation*}
\mathbb{E}\Big[\Big(\sum_{i = 1}^{c_M(\e) s^2} \alpha_i\Big) \mathbbm{1}_{\{K_M(s) \geq c_M(\e) s^2, \h_R(s) \geq m\}}\Big] \geq c_M(\e) s^2  \mathbb{E}\big[X_1 \wedge m\big] \mathbb{P}\Big(K_M(s) \geq c_M(\e)s^2, h_R (s) \geq m\Big). 
\end{equation*}  
We divide both sides by $s^2$ and let $s \to \infty$. Note that $K_M(s)$ is a Poisson random variable with mean $(1- \frac{1}{M})^2 s^2$. By a standard large deviation bound, $\lim_{s \to \infty}\mathbb{P}\big(K_M(s) \geq c_M(\e) s^2  \big)= 1$. Furthermore, by Lemma \ref{eq:nt}, we have $\lim_{s \to \infty}\mathbb{P}(h_R (s) \geq m) = 1$. Hence,
\begin{equation*}
\lim_{s \to \infty}\mathbb{P}\Big(K_M(s) \geq c_M(\e) s^2, h_R(s) \geq m\Big)  = 1. 
\end{equation*}
This implies that 
\begin{equation*}
\liminf_{s \to \infty}\frac{\mathbb{E}\big[ \A_1(s)\big]}{s^2} \geq c_M(\e) \mathbb{E}[X_1 \wedge m].
\end{equation*}	
Letting $M, m \to \infty$ and $\e \to 0$, we get
\begin{equation*}
\liminf_{s \to 
\infty} \frac{\mathbb{E}[\A_1(s)]}{s^2} \geq \frac{1}{1-t}. \qedhere
\end{equation*}	
\end{proof}
To prove Proposition \ref{prop:alphaptlim}, it remains to prove the following lemma. 
\begin{lemma}\label{lem:converge+}
We have $\lim_{ s \to \infty}\mathbb{E}\Big[\big(\frac{1}{s^2} \A_1 (s)  - \frac{1}{1-t}\big)^+\Big] = 0$ where $x^+ := \max(x, 0)$. 
\end{lemma}
\begin{proof}
Since each $\alpha_i(s)$ is stochastically dominated by a geometric random variable, it follows that the random variable $\A_1(s)$ is stochastically dominated by $\sum_{i = 1}^{K(s)} X_i$, where $\{X_i\}_{i = 1}^\infty$ are i.i.d. geometric random variables that are independent of $K(s)$. Hence, it suffices to show that 
\begin{equation}\label{eq:exp+}
\lim_{s \to \infty} \mathbb{E}\bigg[\Big(\frac{1}{s^2} \sum_{i = 1}^{K(s)} X_i - \frac{1}{1-t}\Big)^+ \bigg] = 0.
\end{equation}
We compute
\begin{align*}
\mathbb{E}\bigg[\Big(\frac{1}{s^2} \sum_{i = 1}^{K(s)} X_i - \frac{1}{1-t}\Big)^2 \bigg] &= \mathbb{E}\bigg[\Big(\frac{1}{s^2} \sum_{i= 1}^{K(s)} (X_i - \frac{1}{1-t}) + \frac{1}{1-t} (\frac{K(s)}{s^2} - 1)\Big)^2\bigg]\\
&\leq 2 \mathbb{E}\bigg[\Big(\frac{1}{s^2} \sum_{i = 1}^{K(s)} (X_i - \frac{1}{1-t})\Big)^2\bigg] + \frac{2}{(1-t)^2} \mathbb{E}\Big[(\frac{K(s)}{s^2} - 1)^2\Big] = \mathcal{O}(s^{-2}).
\end{align*}
For the inequality, we use $(a+b)^2 \leq 2(a^2+b^2)$. The last equality follows from a direct computation since $K(s)$ is a Poisson random variable with mean $s^2$. This concludes \eqref{eq:exp+}.
\end{proof}
\begin{proof}[Proof of Proposition \ref{prop:alphaptlim}]
Using Lemma \ref{lem:convergeexp} and \ref{lem:converge+}, we have $\lim_{s \to \infty} \mathbb{E}[|\frac{1}{s^2} \A_1(s) - \frac{1}{1-t}|] = 0$. This implies the convergence in probability. 
\end{proof}
 
\subsection{Proving $\gamma = \frac{2}{\sqrt{1-t}}$}
We proceed to prove that $\gamma = \frac{2}{\sqrt{1-t}}$ by building a connection between the number of $\alpha$-points and the limit shape, using the same argument as in \cite{groeneboom2001ulam}. 
\begin{lemma}
Fix $x, y > 0$. We have
\begin{equation}\label{eq:heightfunc}
N(x, y) = \#\{\alpha \text{ points in } [0, x] \times [0, y] \} - \#\{\beta \text{ points in } [0, x] \times [0, y]\}.
\end{equation}
\end{lemma}
\begin{proof}
For the $t$-PNG model, we can decompose the collection of paths into a number of down-right paths that do not cross each other, see Figure \ref{fig:down-rightpath}.
$N(x, y)$ equals the number of down-right paths that one crosses in the north-east direction from $(0, 0)$ to $(x, y)$. 
\begin{figure}[t]
    \centering
\begin{tikzpicture}
	\begin{scope}
	\draw[->] (0, 0) -- (6, 0); 
	\draw[->] (0, 0) -- (0, 6);
	\draw[dashed] (0, 6) -- (6, 6) -- (6, 0);
	\node at (1.7, 3) {$\olive{P_1}$};
	\node at (2.7, 4) {$\purple{P_2}$};
	\node at (4.7, 5) {$\blue{P_3}$};
	\draw[thick, olive] (1, 6) -- (1, 5) -- (2, 5) -- (2, 2) -- (5, 2) -- (5, 1) -- (6, 1);
	\draw[thick, purple] (2, 6) -- (2, 5) -- (3, 5) -- (3, 3) -- (5, 3) -- (5, 2) -- (6, 2);
		\draw[thick, blue] (5, 6) -- (5, 3) -- (6, 3);
	\end{scope}
	\end{tikzpicture}
    \caption{We decompose the  $t$-PNG model in any rectangle into a sequence of down-right paths. To distinguish between them, we give them different colors (this has nothing to do with the colored model). We label them $P_1, P_2$ and $P_3$. As shown in Figure \ref{fig:def-qpng}, these down-right paths play the role as level lines of the $t$-PNG height function.}
    \label{fig:down-rightpath}
\end{figure}
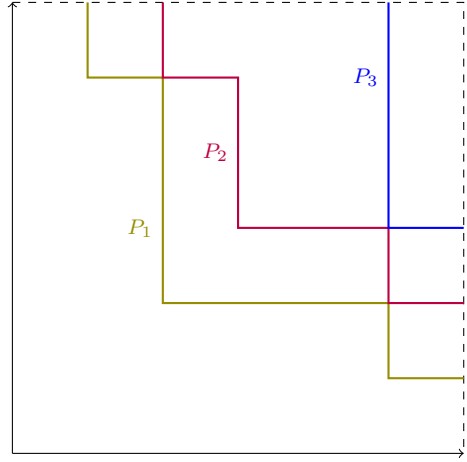
We want to count the number of $\alpha$- and $\beta$- points on each down-right path. To avoid ambiguity at crossing points, which are both $\alpha$- and $\beta$-points,  
we treat the $\beta$-point as lying on the left down-right path passing through the crossing point and the $\alpha$-point as lying on the right down-right path. 
Then, for each down-right path, it is straightforward to see that the difference between the number of $\alpha$-points and the number of $\beta$-points lying on it equals $1$. Hence, the right-hand side of \eqref{eq:heightfunc} is exactly the number of down-right paths that one crosses as going from $(0, 0)$ to $(x, y)$, which equals $N(x, y)$.
\end{proof}

For $x, y > 0$, we denote $V_s (x, y) = s^{-1} N(sx, sy)$. We associate a random measure $\xi_s$ to this random variable by letting 
\begin{equation*}
\xi_s(B) := \int_B dV_s (x, y) = s^{-1} \Big(\# \{\alpha \text{ points in } sB\}  -  \#\{ \beta \text{ points in } sB\}\Big),
\end{equation*}
where $sB=\{(sx, sy): (x, y) \in B\}$.
\bigskip
\\
Let $B = [0, x] \times [0, y]$. We define 
\begin{align*}
V_s^+ (x, y) &= s^{-1} \{\text{number of } \alpha \text{ points in } sB\},\\
V_s^- (x, y) &= s^{-1} \{\text{number of } \beta \text{ points in } sB\},
\end{align*}
and
\begin{equation}\label{eq:Vtilde}
\widetilde{V}_s(x, y) = \int_{[0, x] \times [0, y)} dV_s (u, v). 
\end{equation}
Note that in the integral above, we omit the upper edge of the rectangle $[0, x] \times [0, y]$ but not the lower edge. This is crucial to the establishment of the next lemma, which intrinsically goes back to \cite[Lemma 2.1]{groeneboom2001ulam}. 

\begin{lemma}\label{lem:Vequality}
We have
\begin{equation*}
V_s (x, y)^2 = s^{-1} \big(V_s^+(x, y) + V_s^{-} (x, y)\big) + 2\int_{[0, x] \times [0, y]} \widetilde{V}_s (u, v) dV_s(u, v). 
\end{equation*}
\end{lemma}
\begin{proof}

Consider the down-right paths we cross while going from $(0, 0)$ to $(sx, sy)$ in the north-east direction. We label these paths $P_1,P_2,... , P_m$, where $P_1$ is the down-right path  closest to the origin and $P_m$ is the path farthest from the origin, see Figure \ref{fig:down-rightpath}. Note that we have $V_s(x, y) = \frac{m}{s}$. For an $\alpha$-point $(u,v)$ lying on $P_i$, we have $\widetilde{V}_s(u, v)= \frac{i-1}{s}$. For a $\beta$-point $(u,v)$ lying on $P_i$, we have $\widetilde{V}_s(u,v) = \frac{i}{s}$. 
Recall that $\A_1(x, y)$ is the number
of $\alpha$-points and $\A_2(x, y)$ is the number of $\beta$-points contained in $[0,x]\times [0, y]$. 
We have
\begin{align*}
\int_{[0, x] \times [0, y]} \widetilde{V}_s (u, v) dV_s (u, v) &= s^{-2} \sum_{i = 1}^m \Big((i-1)\#\{\alpha \text{ points on } P_i\} - i\#\{\beta \text{ points on } P_i\}\Big)\\
&=s^{-2} \sum_{i = 1}^m \Big( (i-1) \#\{\alpha \text{ points on } P_i\} -  (i-1)\#\{\beta \text{ points on } P_i\} \Big) - s^{-2} \A_2(sx, sy).
\end{align*} 
For each space-time curve $P_i$, contained in $[0, s] \times 
[0, s]$, we have
\begin{align*}
\#\{\alpha \text{ points on } P_i\} - \#\{\beta \text{ points on } P_i\} = 1
\end{align*}
Using this together with $m = s V_s (x, y)$, we have
\begin{align*}
\int_{[0, x] \times [0, y]} \widetilde{V}_s (u, v) dV_s (u,v) &= (2 s^2)^{-1} m (m-1) -  s^{-2}  \A_2(sx, sy)\\
&=\frac{1}{2} V_s (x, y)^2 - \frac{1}{2}s^{-1} V_s (x, y) - s^{-2} \A_2(sx, sy)\\
&= \frac{1}{2}V_s(x, y)^2 - \frac{1}{2} s^{-1} \big(V_s^+(x, y) + V_s^{-} (x, y)\big).
\end{align*}
In the last equality, we have used $s^{-1}\A_2(sx, sy) = V_s^{-} (x, y)$, and $V_s(x, y) =  V_s^{+} (x, y) -  V_s^{-} (x, y)$. 
\end{proof} 
\begin{proposition}
Recall $\gamma$ from Proposition \ref{prop:LLN}. We have $\gamma = \frac{2}{\sqrt{1-t}}$.
\end{proposition}
\begin{proof}
Note that $V_s^- (x, y) = V_s^+ (x, y) - V_s (x, y)$, and $sV_s^+(x, y) = \A_1 (sx, sy)$. Using this together with Lemma \ref{lem:Vequality}, we have
\begin{equation}\label{eq:Vt}
V_s (x, y)^2 = 2s^{-2} \A_1 (sx, sy) - s^{-1} V_s (x, y) + 2\int_{[0, x] \times [0, t]} \widetilde{V}_s (u, v) dV_s(u, v). 
\end{equation}

By Proposition \ref{prop:LLN}, we have almost surely $\lim_{s \to \infty}V_s (u, v) = \gamma\sqrt{uv}$. 
Hence, we have as $s \to \infty$ (see \cite[page 688]{groeneboom2001ulam} for detail)
\begin{equation*}
\int_{[0, x] \times [0, y]} \widetilde{V}_s (u, v) dV_s (u, v) \to \int_{[0, x] \times [0, y]} \gamma \sqrt{uv} \cdot d(\gamma\sqrt{uv}) = \frac{\gamma^2}{4} xy. 
\end{equation*}
By Proposition \ref{prop:alphaptlim} and scaling, we have $\frac{1}{s^2} \A_1(sx, sy) \overset{p}{\to} \frac{xy}{1-t}$, as $y \to \infty$. Hence, there exists a sequence 
$\{s_n\}_{n \in \Z_{\geq 1}}$ that goes to infinity and satisfies $\lim_{n \to \infty} \frac{1}{s_n^2} \A_1(s_n x, s_n y) = \frac{xy}{1-t}$ almost surely.  Take $s = s_n$ and let $n \to \infty$ on both sides of \eqref{eq:Vt}. The left-hand side converges to $\gamma^2 xy$ and the right-hand side converges to $\frac{2xy}{1-t} + \frac{\gamma^2}{2} xy$. 
Hence, we have
$\gamma^2 xy = \frac{2}{1-t} xy + \frac{1}{2} \gamma^2 xy.$ Using this together with Lemma \ref{upper-bound} (which shows that $\gamma$ is finite),
we have $\gamma = \frac{2}{\sqrt{1-t}}$. 
\end{proof}

\begin{rmk}
Since all of the terms containing $V_s$ in \eqref{eq:Vt} converge almost  surely as $s \to \infty$, so does $\A(sx, sy)$. Hence, we can strengthen Proposition \ref{prop:alphaptlim} to almost sure convergence.    
\end{rmk}
\begin{proof}[Proof of Theorem \ref{hydrodynamic_limit}]
The theorem is a direct consequence of Proposition \ref{prop:LLN} and Proposition \ref{prop:alphaptlim}. The convergence  $\frac{N (x,y)}{\sqrt{xy}} \overset{p}{\to} \frac{2}{\sqrt{1-t}}$ follows from Lemma \ref{lem:scaling}.
\end{proof}

\begin{appendix}
\section{Configurations for the two-colored $t$-PNG model}
\label{sec:twocolorfig}
In this section, we draw the two-colored configurations that have non-zero weights. In other words, we draw all of the two-colored configurations with $\bi, \bj, \bk, \bl \in \{0, 1\}^2$ such that $\LL^2(\bi, \bj; \bk, \bl) \neq 0$.

\begin{figure}[ht]
\centering
\begin{tikzpicture}[scale = 1]
\begin{scope}[xshift = 0cm, yshift = 4cm]
\node at (0, -0.8) {$1$};
\end{scope}
\begin{scope}[xshift = 2cm, yshift = 4cm]
\draw[red, thick] (-0.5, 0) -- (0.5, 0);
\draw[red, thick] (0, -0.5) -- (0, 0.5);
\node at (0, -0.8) {$t$};
\end{scope}
\begin{scope}[xshift = 4cm, yshift = 4cm]
\draw[red, thick] (-0.5, 0) -- (0, 0);
\draw[red, thick] (0, -0.5) -- (0, 0);
\node at (0, -0.8) {$1-t$};
\end{scope}
\begin{scope}[xshift = 6cm, yshift = 4cm]
\draw[red, thick] (-0.5, 0) -- (0.5, 0);
\node at (0, -0.8) {$1$};
\end{scope}
\begin{scope}[xshift = 8cm, yshift = 4cm]
\draw[red, thick] (0, -0.5) -- (0, 0.5);
\node at (0, -0.8) {$1$};
\end{scope}
\begin{scope}[xshift = 0cm, yshift = 2cm]
\draw[blue, thick] (-0.5, 0) -- (0.5, 0);
\draw[blue, thick] (0, -0.5) -- (0, 0.5);
\node at (0, -0.8) {$t$};
\end{scope}
\begin{scope}[xshift = 2cm, yshift = 2cm]
\draw[blue, thick] (-0.5, 0) -- (0, 0);
\draw[blue, thick] (0, -0.5) -- (0, 0);
\node at (0, -0.8) {$1-t$};
\end{scope}
\begin{scope}[xshift = 4cm, yshift = 2cm]
\draw[blue, thick] (-0.5, 0) -- (0.5, 0);
\node at (0, -0.8) {$1$};
\end{scope}
\begin{scope}[xshift = 6cm, yshift = 2cm]
\draw[blue, thick] (0, -0.5) -- (0, 0.5);
\node at (0, -0.8) {$1$};
\end{scope}
\begin{scope}[xshift = 0cm, yshift = 0cm]
\draw[thick, red] (-0.1, -0.5) -- (-0.1, 0.5);
\draw[thick, blue] (0, -0.5) -- (0, 0.5);
\node at (0, -0.8) {$1$};
\end{scope}
\begin{scope}[xshift = 2cm]
\draw[thick, red] (-0.5, 0) -- (0.5, 0);
\draw[thick, blue] (-0.5, -0.1) -- (0.5, -0.1);
\node at (0, -0.8) {$1$};
\end{scope}
\begin{scope}[xshift = 4cm]
\draw[red, thick] (-0.5, 0) -- (0.5, 0);
\draw[blue, thick] (0, -0.5) -- (0, 0.5);
\node at (0, -0.8) {$t$};
\end{scope}
\begin{scope}[yshift = 0cm, xshift = 6cm]
\draw[red, thick] (-0.5, 0) -- (0.5, 0);
\draw[blue, thick] (0, -0.5) -- (0, -0.1);
\draw[blue, thick] (0, -0.1) -- (0.5, -0.1); 
\node at (0, -0.8) {$1 - t$};
\end{scope}
\begin{scope}[xshift = 8cm]
\draw[blue, thick] (-0.5, 0) -- (0.5, 0);
\draw[red, thick] (0, -0.5) -- (0, 0.5);
\node at (0, -0.8) {$t$};
\end{scope}
\begin{scope}[xshift = 10cm]
\draw[blue, thick] (-0.5, 0) -- (-0.1, 0);
\draw[red, thick] (0, -0.5) -- (0, 0.5);
\draw[blue, thick] (-0.1, 0) -- (-0.1, 0.5);
\node at (0, -0.8) {$1 - t$};
\end{scope}
\begin{scope}[yshift = -2cm]
\draw[blue, thick] (-0.5, 0) -- (0.5, 0);
\draw[red, thick] (-0.1, -0.5) -- (-0.1, 0.5);
\draw[blue, thick] (0, -0.5) -- (0, 0.5);
\node at (0, -0.8) {$1$};
\end{scope}
\begin{scope}[yshift = -2cm, xshift = 2cm]
\draw[blue, thick] (-0.5, 0) -- (0.5, 0);
\draw[red, thick] (-0.5, -0.1) -- (0.5, -0.1);
\draw[blue, thick] (0, -0.5) -- (0, 0.5);
\node at (0, -0.8) {$1$};
\end{scope}
\begin{scope}[yshift = -2cm, xshift = 4cm]
\draw[red, thick] (-0.5, 0) -- (0.5, 0);
\draw[red, thick] (-0.1, -0.5) -- (-0.1, 0.5);
\draw[blue, thick] (0, -0.5) -- (0, 0.5);
\node at (0, -0.8) {$t$};
\end{scope}
\begin{scope}[yshift = -2cm, xshift = 6cm]
\draw[red, thick] (-0.5, 0) -- (-0.1, 0);
\draw[red, thick] (-0.1, -0.5) -- (-0.1, 0);
\draw[blue, thick] (0, -0.5) -- (0, 0);
\draw[blue, thick] (0, 0) -- (0.5, 0);
\node at (0, -0.8) {$1 - t$};
\end{scope}
\begin{scope}[yshift = -2cm, xshift = 8cm]
\draw[blue, thick] (-0.5, 0) -- (0.5, 0);
\draw[red, thick] (-0.5, -0.1) -- (0.5, -0.1);
\draw[red, thick] (0, -0.5) -- (0, 0.5);
\node at (0, -0.8) {$t$};
\end{scope}
\begin{scope}[yshift = -2cm, xshift = 10cm]
\draw[blue, thick] (-0.5, 0) -- (0, 0) -- (0, 0.5);
\draw[red, thick] (-0.5, -0.1) -- (0, -0.1);
\draw[red, thick] (0, -0.5) -- (0, -0.1);
\node at (0, -0.8) {$1-t$};
\end{scope}
\begin{scope}[yshift = -4cm, xshift = 0cm]
\draw[red, thick] (-0.1, -0.5) -- (-0.1, 0.5);
\draw[red, thick] (-0.5, -0.1) -- (0.5, -0.1);
\draw[red, thick] (-0.1, -0.5) -- (-0.1, 0.5);
\draw[blue, thick] (0, -0.5) -- (0, 0.5);
\draw[blue, thick] (-0.5, 0) -- (0.5, 0);
\node at (0, -0.8) {$t$};
\end{scope}
\begin{scope}[yshift = -4cm, xshift = 2cm]
\draw[red, thick] (-0.1, -0.5) -- (-0.1, -0.1);
\draw[red, thick] (-0.5, -0.1) -- (-0.1, -0.1);
\draw[blue, thick] (0, -0.5) -- (0, 0);
\draw[blue, thick] (-0.5, 0) -- (0, 0);
\node at (0, -0.8) {$1 - t$};
\end{scope}
\end{tikzpicture}
\caption{Let red be the color $1$, and let blue be the color $2$. The above figure shows all the configurations with non-zero weights in the two-colored $t$-PNG model given by Definition \ref{def:nLmatrix}. Note that some of the configurations can be obtained from the two-colored S6V model by horizontal complementation of the red lines (i.e., deleting existing horizontal red lines and placing horizontal red lines where there aren't any). However, not all the configurations (e.g., the ones on the last row) can be obtained in this way.}
\label{fig:twocolor}
\end{figure}
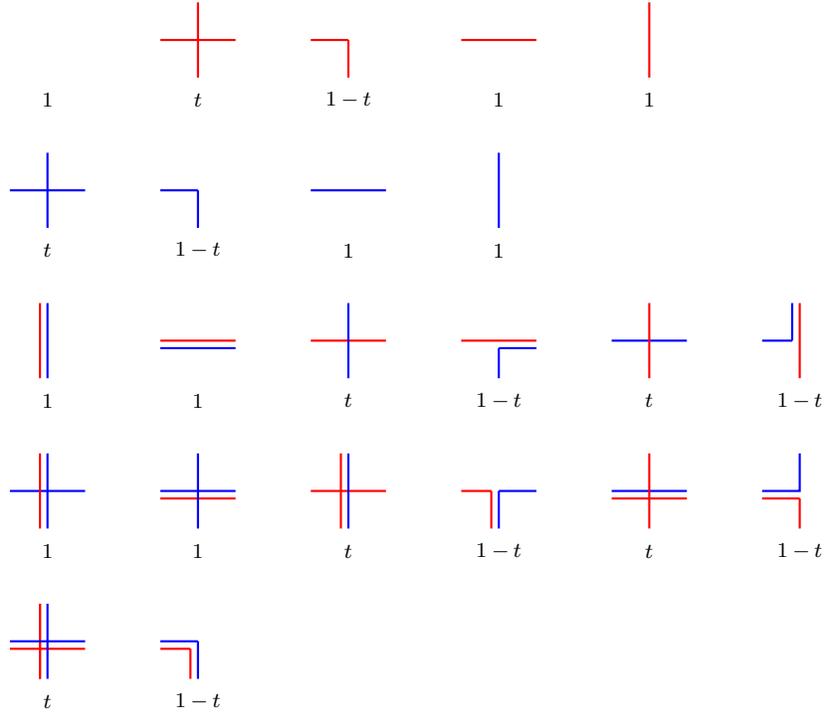
\newpage
\section{Computing $G^*$}\label{sec:generators}
In this section, we prove the following lemma.
\begin{lemma} \label{lem:dualGen}
We have
$$
    G^*g(y) = \sum_{i=1}^{\infty}t^{i-1}(1-t)\int_0^{T_1}(g(\mathcal{L}^i_sy) - g(x))ds + \sum_{i=1}^{\infty}\frac{t^{i-1}}{\lambda}(g(\mathcal{R}^iy) - g(y)).
$$
\end{lemma}
\begin{proof}
Recall that $\mu$ is the measure induced on $E$ by a Poisson process with intensity $\lambda$. Let $\mu_n$ be the restriction of $\mu$ to $E_n$, where $E_n$ consists of configurations with $n$ particles and where $E_0 = \{\emptyset\}$ consists of the empty configuration. Then we have 
$\mu_n(dx) = \lambda^ne^{-\lambda T_1}dx.$

We define 
\begin{equation}\label{eq:G+}
G_{+} f = Gf + \left(\frac{1}{\lambda(1-t)} + T_1\right)f. 
\end{equation}

For simplicity, we compute the dual of $G_{+}$ which will then yield the dual of $G$. For $f,g \in C_0(E)$ we have 
\begin{align*}
&\int_{E} G_{+} f(x) g(x) \mu(d x) =\sum_{n=0}^{\infty} \int_{E_{n}} G_{+} f(x) g(x) \mu_{n}(d x) \\
&=e^{-\lambda T_{1}} \sum_{n=0}^{\infty}\left[\lambda^{n} \int_{E_{n}} \int_{0}^{T_{1}}\sum_{i=1}^{\infty}t^{i-1}(1-t) f\left(\mathcal{R}^i_{z} x\right) g(x) d z d x+\lambda^{n-1}\int_{E_{n}}\sum_{i=1}^{\infty}t^{i-1} f(\mathcal{L}^i x) g(x) d x\right]\\
&= e^{-\lambda T_1} \sum_{n = 0}^\infty A_n.
\end{align*}

Suppose that $x = (x_1, \ldots, x_n) \in E_n, z \in [0, T_1]$, and there are $m-1$ particles in $x$ to the left of $z$. Let $j(z) = n-m-1$ equal the number of particles to the right of $z$. We can decompose $A_n$ into 
\begin{align}
A_n & = \lambda^n \sum_{i=1}^{\infty}t^{i-1}(1-t)\int_{\left\{x \in E_{n}, z \in [0,T_1], i \leq j(z)\right\}} f\left(\mathcal{R}^i_z x\right)g(x) d x d z \label{eq:G1}\\
&\quad +\lambda^n \sum_{i=1}^{\infty}t^{i-1}(1-t)\int_{\left\{x \in E_{n}, z \in [0,T_1], i > j(z)\right\}} f\left(\mathcal{R}^i_z x\right)g(x) d x d z  \label{eq:G2}\\
&\quad +\lambda^{n-1}\sum_{i=1}^nt^{i-1}\int_{E_n} f(\mathcal{L}^ix)g(x)dx  \label{eq:G3} \\
&\quad +\lambda^{n-1}\sum_{i=n+1}^{\infty}t^{i-1}\int_{E_n} f(\mathcal{L}^ix)g(x)dx.  \label{eq:G4}
\end{align}

We will perform a change of variables by setting $y$ to equal the input to the function $f$. In other words, we will set $y = \mathcal{R}^i_z x$ or $y = \mathcal{L}^ix$ in the integrals above. We then need to solve for $x$ in terms of $y$ where $y$ belongs to the dual system. Let $k(s)$ equal the number of particles to the left of $s$ in the dual system. Then we have the following four cases: 
\begin{enumerate}[leftmargin = 2em]
    \item $y = \mathcal{R}_z^ix \in E_n$ for $i \leq j(z)$.
    In this case, $x = \mathcal{L}_s^iy$ where $s = x_{m+i-1}$.
    
    \item $y = \mathcal{R}_z^ix \in E_{n+1}$ for $i > j(z)$.
    In this case, $x = \mathcal{R}^{j(z)+1}y.$
    
    \item $y =  \mathcal{L}^ix \in E_{n-1}$ for $i \leq n$.
    In this case, $x = \mathcal{L}_s^ky$ for $s = x_{i}$ and for all $k > k(s)$.
    
    \item $y = \mathcal{L}^ix \in E_n$ for $i > n$.
    In this case, $x = \mathcal{R}^ky = y$ for all $k > n$.
    
\end{enumerate}

Therefore, when we perform the change of variables, the first and third terms \eqref{eq:G1} and \eqref{eq:G3} will correspond to the operators $\mathcal{L}^i_s$ and the second and fourth terms \eqref{eq:G2} and \eqref{eq:G4} will correspond to the operator $\mathcal{R}^i$. We compute the change of variables explicitly on each of the four terms. 
\bigskip
\\
\textbf{First term:} We have
\begin{align*}
&\lambda^n \sum_{i=1}^{\infty}t^{i-1}(1-t)\int_{\left\{x \in E_{n}, z \in [0,T_1], i \leq j(z)\right\}} f\left(\mathcal{R}^i_z x\right)g(x) d x d z\\
&=\lambda^n \sum_{i=1}^{\infty}t^{i-1}(1-t)\int_{\left\{y \in E_{n}, s \in [0, T_1], i \leq k(s)\right\}} f\left(y\right)g(\mathcal{L}_s^i y) dy ds.
\end{align*}
\textbf{Second term:} We have
\begin{align*}
&\lambda^n \sum_{i=1}^{\infty}t^{i-1}(1-t)\int_{\left\{x \in E_{n}, z \in [0,T_1], i > j(z)\right\}} f\left(\mathcal{R}^i_z x\right)g(x) d x d z\\
&= \lambda^n \sum_{i=1}^{\infty}\sum_{j=0}^{i-1}t^{i-1}(1-t)\int_{\left\{x \in E_{n}, z \in [0,T_1], j(z) = j\right\}} 
f\left(\mathcal{R}^i_z x\right)g(x) d x d z\\
&= \lambda^n \sum_{i=1}^{\infty}\sum_{j=0}^{i-1}t^{i-1}(1-t)\int_{\left\{y \in E_{n+1}\right\}} 
f\left(y\right)g(\mathcal{R}^{j+1}y) dy.
\end{align*}
We sum over $i$ and shift the index $j$ by one to get 
$$\lambda^n \sum_{j=1}^{\infty}t^{j-1}(1-t)\int_{\{y \in E_{n+1}\}} f\left(y\right)g(\mathcal{R}^{j}y) dy.
$$
\textbf{Third term:} We have
\begin{align*}
\lambda^{n-1}\sum_{i=1}^nt^{i-1}\int_{E_n} f(\mathcal{L}^ix)g(x)dx &=\lambda^{n-1}\sum_{i=1}^n \sum_{j=i}^{\infty}t^{j-1}(1-t)\int_{E_n} f(\mathcal{L}^ix)g(x)dx \\
&=\lambda^{n-1}\sum_{i=1}^n \sum_{j=i}^{\infty}t^{j-1}(1-t)\int_{\{y \in E_{n-1}, s \in [0, T_1], y_{i-1} \leq  s \leq y_i\}} f(y)g(\mathcal{L}^{j}_sy)dy ds \\
&=\lambda^{n-1} \sum_{j=1}^{\infty}t^{j-1}(1-t) \int_{\{y \in E_{n-1}, s \in [0, T_1],  j > k(s)\}} f(y)g(\mathcal{L}^{j}_sy)dy ds.
\end{align*}
\textbf{Fourth term:} We have
\begin{align*}
\lambda^{n-1}\sum_{i=n+1}^{\infty}t^{i-1}\int_{E_n} f(\mathcal{L}^ix)g(x)dx &= \lambda^{n-1}\sum_{i=n+1}^{\infty}t^{i-1}\int_{E_n} f(y)g(\mathcal{R}^{i}y)dy.
\end{align*}
We collect the four terms and relabel the index $j$ by $i$ if necessary to obtain
\begin{align*} 
A_n &= \lambda^n \sum_{i=1}^{\infty}t^{i-1}(1-t)\int_{\left\{y \in E_{n}, s \in [0, T_1], i \leq k(s)\right\}} f\left(y\right)g(\mathcal{L}_s^i y) dy d s\\
&\quad +\lambda^n \sum_{i=1}^{\infty}t^{i-1}(1-t)\int_{E_{n+1}} f\left(y\right)g(\mathcal{R}^{i}y) dy\\
&\quad +\lambda^{n-1} \sum_{i=1}^{\infty}t^{i-1}(1-t) \int_{\{y \in E_{n-1}, s \in [0, T_1],  i > k(s)\}} f(y)g(\mathcal{L}^{i}_sy)dy ds\\
&\quad +\lambda^{n-1}\sum_{i=n+1}^{\infty}t^{i-1}\int_{E_n} f(y)g(\mathcal{R}^{i}y)dy. 
\end{align*}
Finally, we re-index each term so that all the integrals are over $E_n$ instead of $E_{n-1}$ or $E_{n+1}$:
\begin{align*}
A_n &= \lambda^n \sum_{i=1}^{\infty}t^{i-1}(1-t)\int_{\left\{y \in E_{n}, s \in [0, T_1], i \leq k(s)\right\}} f\left(y\right)g(\mathcal{L}_s^i y) dy d s \\
&\quad +\lambda^{n-1} \sum_{i=1}^{\infty}t^{i-1}(1-t)\int_{E_{n}} f\left(y\right)g(\mathcal{R}^{i}y) dy \\
&\quad +\lambda^{n} \sum_{i=1}^{\infty}t^{i-1}(1-t) \int_{\{y \in E_{n}, s \in [0, T_1],  i > k(s)\}} f(y)g(\mathcal{L}^{i}_sy)dy ds \\
&\quad +\lambda^{n-1}\sum_{i=n+1}^{\infty}t^{i-1}\int_{E_n} f(y)g(\mathcal{R}^{i}y)dy.
\end{align*}
Multiplying everything back by $e^{-\lambda T_1}$ and summing over $n$, we get
\begin{align*}
&\int_{E} G_{+} f(x) g(x) \mu(d x)  = e^{-\lambda T_1} \sum_{n = 0}^\infty A_n\\
&=e^{-\lambda T_{1}} \sum_{n=0}^{\infty}\left[\lambda^{n} \int_{E_{n}} \int_{0}^{T_{1}}\sum_{i=1}^{\infty}t^{i-1}(1-t) f\left(y\right) g(\mathcal{L}^i_sy) d s d y+\lambda^{n-1}\int_{E_{n}}\sum_{i=1}^{\infty}t^{i-1} f(y) g(\mathcal{R}^i y) d y\right]\\
&=\int_E f(y) G_+^*g(y)\mu(dy),
\end{align*}
where $G_+^*$ is the dual of $G_+$. 
By \eqref{eq:G+}, we have
$$
G^*g =  G_+^*g - \Big(\frac{1}{\lambda(1-t)}+ T_1\Big)g,
$$
which concludes the proof.

\end{proof}

\section{Simulation of the $t$-PNG model}
\label{sec:simulations}
In the following, we provide simulations of the $t$-PNG model on the square $[0, 12] \times [0, 12]$ for different values of $t$.
\begin{figure}[ht]
\includegraphics[scale = 0.4]{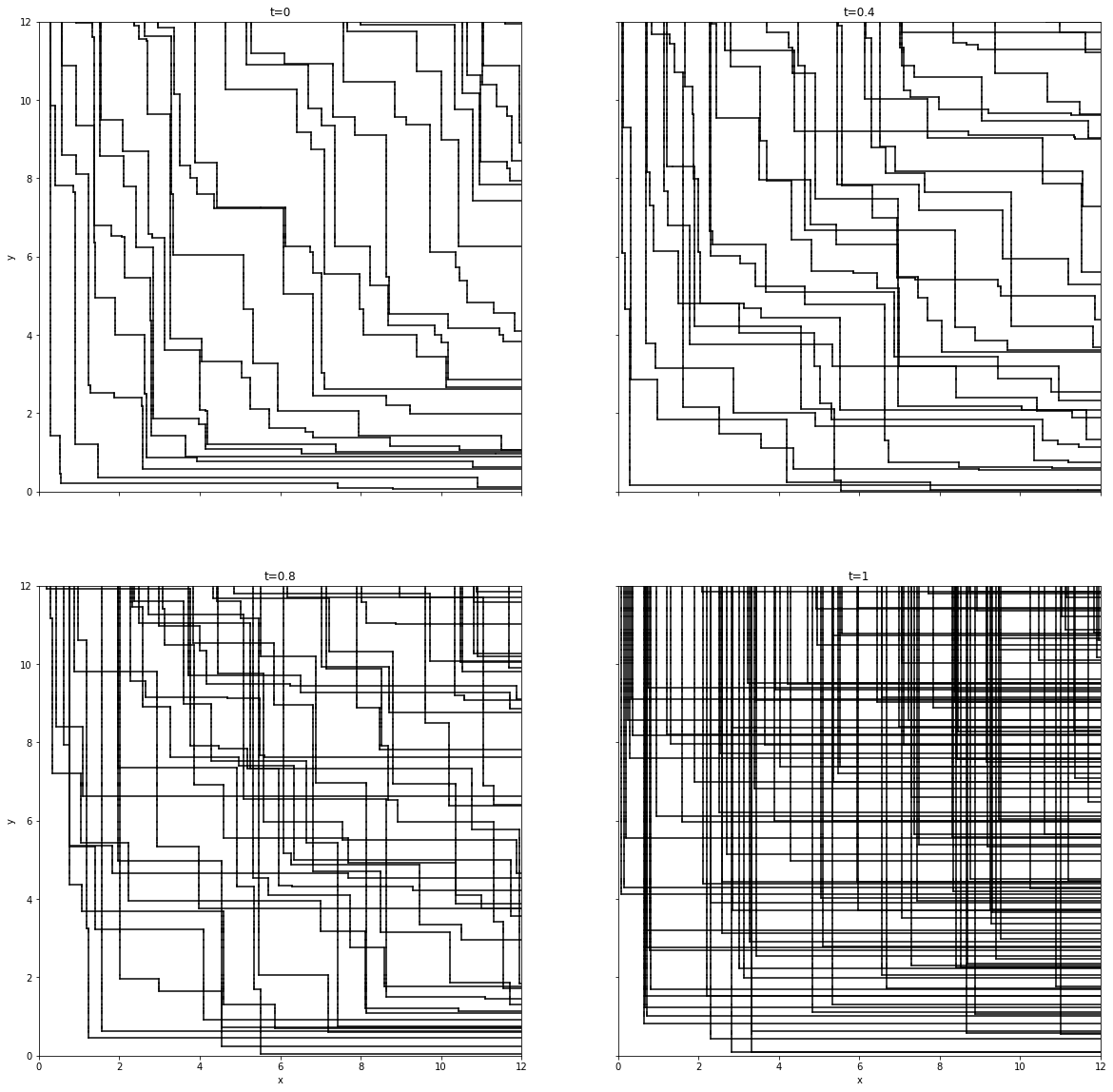}
\centering
\caption{Samplings of the $t$-PNG model for four different values of $t$
.}
\label{fig:simulations}
\end{figure}
\end{appendix}

\begin{acks}[Acknowledgments]
The authors thank Amol Aggarwal, Ivan Corwin, Pablo Ferrari, and Firas Rassoul-Agha for the helpful discussion. We thank the anonymous referees for their helpful comments.
\end{acks}

\begin{funding}
The authors acknowledge support from NSF DMS-1928930 during
their participation in the program “Universality and Integrability in Random Matrix Theory and
Interacting Particle Systems” hosted by the Mathematical Sciences Research Institute in Berkeley, California during the Fall semester of 2021. Hindy Drillick was supported by the National Science Foundation Graduate Research Fellowship under Grant No. DGE-1644869. 
%
\end{funding}


\bibliographystyle{imsart-number} 
\bibliography{refs.bib}       


\end{document}